\documentclass[11pt]{article} % use larger type; default would be 10pt

\usepackage[utf8]{inputenc} % set input encoding (not needed with XeLaTeX)

\usepackage{graphicx} % support the \includegraphics command and options

\usepackage{booktabs} % for much better looking tables
\usepackage{array} % for better arrays (eg matrices) in maths
\usepackage{paralist} % very flexible & customisable lists (eg. enumerate/itemize, etc.)
\usepackage{verbatim} % adds environment for commenting out blocks of text & for better verbatim
\usepackage{subfig} % make it possible to include more than one captioned figure/table in a single float
\usepackage[T1]{fontenc}
\usepackage[utf8]{inputenc}
\usepackage{lmodern}
\usepackage[a4paper]{geometry}
\usepackage[american]{babel}
\usepackage{stix}
\usepackage{mathtools}
\usepackage{amsthm}
\usepackage[backend=biber,style=numeric,doi=false,isbn=false,url=false,eprint=false]{biblatex}
\usepackage{csquotes}
\usepackage{listings}
\usepackage{algorithm}
\usepackage{algpseudocode}
\usepackage{bbm}
\usepackage{nicefrac}
\usepackage{titletoc}
\usepackage{tikz}

\newtheorem{theorem}{Theorem}
\newtheorem{lemma}[theorem]{Lemma}

\newtheorem{proposition}{Proposition}
\newtheorem{definition}{Definition}

\theoremstyle{definition}

\theoremstyle{remark}
\newtheorem*{remark}{Remark}

\usepackage{fancyhdr} % This should be set AFTER setting up the page geometry
\usepackage{sectsty}
\allsectionsfont{\sffamily\mdseries\upshape} % (See the fntguide.pdf for font help)
\usepackage[nottoc,notlof,notlot]{tocbibind} % Put the bibliography in the ToC
\usepackage[titles,subfigure]{tocloft} % Alter the style of the Table of Contents

\newcommand{\R}{\mathbbm{R}}
\newcommand{\N}{\mathbbm{N}}
\newcommand{\Z}{\mathbbm{Z}}
\newcommand{\E}{\mathbf{E}}
\newcommand{\V}{\mathcal{V}}
\newcommand{\Sur}{\mathcal{S}}
\newcommand{\either}{\#}

\providecommand{\keywords}[1]
{
	\small	
	\textbf{\textit{Keywords---}} #1
}

\DeclareMathOperator{\Int}{Int}
\DeclareMathOperator{\Type}{Type}
\DeclareMathOperator{\Card}{Card}

\DeclareMathOperator{\Label}{Label}

\newcommand{\norm}[1]{\| #1 \|}

\newcommand{\ceil}[1]{\left \lceil #1 \right \rceil}

\title{Liquid-Gas phase transition for Gibbs point process with Quermass interaction}
\author{D.~Dereudre, C.~Renaud~Chan}
\begin{document}
	
	\maketitle
	
	\begin{abstract}
		We prove the existence of a liquid-gas phase transition for continuous Gibbs point process in $\R^d$ with Quermass interaction. The Hamiltonian we consider is a linear combination of the volume $\mathcal{V}$, the surface measure $\mathcal{S}$ and the Euler-Poincaré characteristic $\chi$ of a halo of particles (i.e. an union of balls centred at the positions of particles). We show the non-uniqueness of infinite volume Gibbs measures for special values of activity and temperature, provided that the temperature is low enough. Moreover we show the non-differentiability of the pressure at these critical points. Our main tool is an adaptation of the Pirogov-Sina\"i-Zahradnik theory for continuous systems with interaction exhibiting a saturation property.  
	\end{abstract}
	
	%TC:ignore
	\keywords{Gibbs measure, DLR equations, Widom-Rowlinson model, Pirogov-Sinaï-Zahradnik theory, cluster expansion}
	%TC:endignore

	\section{Introduction}
	
	Spatial point processes with interaction are models describing the locations of objects, particles in a domain and for which the interaction can be of different nature; attractive, repulsive or a mix of both at different scales of distance between the points. The most popular point process is surely the Poisson point process \cite{last_penrose} which describes random objects without interaction between each other. There are different models for point processes with interaction as for instance Cox point processes \cite{last_penrose}, determinantal point processes \cite{macchi}, zeros of random polynomials or analytic functions \cite{Hough}, Gibbs point processes \cite{MiniCours, ruelle_livre}, etc. Among the field of applications of such models we have plant ecology, telecommunication, astronomy, data science  and statistical physics. 
	
	A large class of point process coming from statistical physics is the family of Gibbs point processes. The finite volume Gibbs point process on a bounded set $\Delta \subset \R^d$ is defined via an unnormalized density, with respect to the Poisson point process in $\Delta$, of the form $z^{N_\Delta} e^{- \beta H}$. The parameters $z$ and $\beta$ are positive numbers (respectively called the activity and the inverse temperature), $N_\Delta$ is the number of points inside $\Delta$ and $H$ an energy function (also called the Hamiltonian). By taking the thermodynamic limit (i.e $\Delta\to \R^d$) we obtain the infinite volume Gibbs point processes which are characterized by the equilibrium equations also known as the Dobrushin-Lanford-Ruelle (DLR) equations.
	
	The type of interaction we study in this paper is the Quermass interaction. In general, the energy function is defined as a linear combination of the $d+1$ Minkowski functionals of the halo of the configuration, which is the union of closed balls centred at the position of the particles with random radii. In this paper, we will consider only the volume, measure of the surface and the Euler-Poincaré characteristic. This type of interaction is a natural extension of the Widom-Rowlinson interaction for penetrable spheres \cite{WidomRowlinson}. Hadwiger's characterization Theorem \cite{Hadwiger} ensures that any functional $F$ on finite unions of convex compact spaces, continuous for the Haussdorff topology, invariant under isometric transformations and additive (i.e. for $A,B \subset \R^d$ such that $A \cup B$ is convex, $F(A \cup B) = F(A) + F(B)- F(A \cap B)$ ) can be written as a linear combination of Minkowski functionals. This representation justifies the study of the Quermass interaction for modelling a large class of morphological interactions. Furthermore, this family of energy function is used to describe the physics of micro-emulsions and complex fluids \cite{likos_emulsion,Mecke_complex_fluid,mecke_integralgeometrie}. The stability condition in the sense of Ruelle has been established in \cite{Kendall} and the existence of the infinite volume Gibbs measures with random radii has been tackled in \cite{dereudre_existence_quermass}. 
	
	In general a phase transition phenomenon occurs when for some special value of $\beta$ and $z$, the boundary conditions at infinity  influence some macroscopic statistics in the bulk of the system. A phase transition phenomenon can also be defined by uniqueness/non-uniqueness of solutions to the DLR equilibrium equations. More specifically we call a liquid-gas or first order phase transition when several solutions have different intensities. The Gibbs measure with the lowest density is associated to the distribution of a pure gas phase and in contrary the highest density to the distribution of a pure liquid phase. 
	
	If we consider models of particles with different spins, there exists an abundance of results on phase transition mainly based on the existence of a dominant spin. We can cite for instance the phase transition results on the continuous Potts model \cite{georgii_hagstrom} and the non symmetrical multiple colour Widom-Rowlinson model \cite{bricmont_kuroda}. Without spin, we need to rely only on the self arrangement, the geometry and the density of particles for proving the phase transition. It is more difficult and there exist only few known results in this setting. The first result of phase transition without spin has been proved for the Widom-Rowlinson model. It is a Gibbs point process where the Hamiltonian is given by  $ H(\omega) = \mathcal{V}(\cup_{x \in \omega} B(x,R))$ where $\mathcal{V}$ is the volume functional. The first proofs are due to Widom-Rowlinson \cite{WidomRowlinson} and Ruelle \cite{RuelleWR} and use extensively the symmetry of the associated two-colour Widom-Rowlinson process. They prove the existence of a critical activity $z_c$ such that for $z>z_c$ and $\beta = z$ the liquid-gas phase transition occurs. Later another proof using Fortuin-Kasteleyn representation has been given \cite{2ChayesKotecky}. Recently, the full phase diagram has been almost entirely mapped by proving an OSSS inequality \cite{HoudebertDereudre}. Outside of a region around the critical point $(z_c, z_c)$, the phase transition occurs if and only if $\beta = z$ for $z > z_c$. Numerical studies have shown that the unmapped region is actually small \cite{Houdebert}. Among continuous models without any form of special symmetry, there is the beautiful result of liquid gas phase transition for an attractive pair and repulsive four body potential of the Kac type \cite{LebowitzMazelPresutti}. They manage to prove, using the Pirogov-Sina\"i-Zahradnik theory (PSZ), the existence of a phase transition for finite but long range interaction compared to typical distance between particles. Indeed the finite range interaction is obtained as a perturbation of the mean field interaction. In a more recent work using a similar strategy, it has been proven that the liquid gas phase transitions persists if a hard-core interaction is considered \cite{Pulvirenti}. Until now, these results were the only proofs of phase transition for continuum systems without spins. 
	
	In the present paper, we are interested in the phase transition phenomenon for the Quermass interaction with the particular form $H(\omega) = \mathcal{V}(\cup_{x \in \omega } B(x,R)) + \theta_1 \mathcal{S}(\cup_{x \in \omega } B(x,R)) - \theta_2 \chi(\cup_{x \in \omega } B(x,R))$ where $\mathcal{S}$ is the measure of the surface and $\chi$ is the Euler-Poincaré characteristic. Contrary to the Widom Rowlinson model we cannot benefit from any symmetry of a two spin process due to the contribution of the surface measure and Euler-Poincaré characteristic. Moreover, in our work we do not use mean-field approximation and the range of interaction is at the same order than the distance between particles. We achieve to prove, for some parameter $\theta_1$ and $\theta_2$, the existence of two infinite Gibbs point processes with distinct intensities when $\beta$ is sufficiently large (i.e. at low temperature) and the activity $z$ equals to a  critical activity $z_\beta^c$. The main tool of our proof is an adaptation of the Pirogov-Sinaï-Zahradnik (PSZ) theory \cite{Zahradnik, PirogovSinai, PirogovSinai2} for continuous systems with interaction satisfying a saturation property. For a modern presentation of the PSZ theory  we advise the lecture of Chapter 7 of \cite{Velenik}. Let us describe succinctly what the saturation property is. First we consider a discretisation of the space $\R^d$ with cubic tiles. Then we introduce two types of "ground states", which are not usual ground states as minimisers of the energy, but rather as extreme idealizations of the point process. One ground state is the empty state and the other corresponds to a dense and homogeneous distribution of particles (i.e. at least one particle in any tile). These states have the property to provide explicit and tractable energy for an extra point emerging in a ground state. That is what we call saturation property.  Using the PSZ theory  we are able to show that for $\beta$ large enough and at the critical activity $z_\beta^c$ the pressure of the model is not differentiable with respect to $z$. Furthermore we construct two different Gibbs measures with the "ground states" as boundary conditions and we make the relation between the intensity of these Gibbs measures and the left and right derivatives of the pressure. We believe that our method is robust enough to deal with other interactions with  similar saturation property. 
	
	Our paper is organized as follows. In Section \ref{Section2}, we introduce the notations, the Quermass interaction and the associated Gibbs point processes. In Section \ref{Section3}, we present the main results of the paper. In Section \ref{Section4}, we develop the tools and we prove the main results. Annex A contains results on Cluster Expansion and Annex B contains the technical proof of Proposition \ref{prop_tau_stability_truncated_weights} at the heart of the PSZ theory.
	
	\tableofcontents

	\section{Quermass interaction model} \label{Section2}
	
	\subsection{State spaces and notations}
	
	We denote by $\mathcal{B}_b(\R^d)$ the set of bounded Borel sets of $\R^d$ with positive Lebesgue measure. For any sets $A$ and $B$ in $\mathcal{B}_b(\R^d)$, $A \oplus B$ stands for the Minkowski sum of these sets; i.e. $A \oplus B := \{x+y, \forall x \in A, \forall y \in B\}$. Let $R_0, R_1$ be two real numbers such that $0<R_0 \leq R_1$.  We denote by $E$ the state space of a single marked point defined as $\R^d \times [R_0, R_1]$. For any $(x,R)\in\E$, the first coordinate $x$ is for the location of the point and the second coordinate $R$ is the mark representing the radius of a ball. For any set $\Delta \in \mathcal{B}_b(\R^d)$, $E_\Delta$ is the local state space $\Delta \times [R_0, R_1]$. A configuration of marked points $\omega$ is a locally finite set in $E$; i.e. $N_\Delta(\omega) := \Card(\omega \cap E_\Delta)$ is finite for any $\Delta \in \mathcal{B}_b(\R^d)$. We denote by $\Omega$ the set of all marked point configurations and by $\Omega_f$ the subset of finite configurations. For any $\omega\in\Omega$, its projection in $\Delta\subset \R^d$ is defined by $\omega_\Delta:=\omega \cap E_\Delta$. As usual we equipped the state space $\Omega$ with the $\sigma$-algebra $\mathcal{F}$ generated by the counting functions on $E$. The halo of a configuration $\omega\in \Omega$ is defined as 
	\begin{equation}
		L(\omega) = \bigcup_{(x,R) \in \omega} B(x,R) 
	\end{equation}
	where $B(x,R)$ is the closed ball centred at  $x$ with radius $R$.
	
	\subsection{Interaction}
	
	Let us introduce the Quermass interaction as in \cite{Kendall}. Usually it is defined as a linear combination of $d+1$ Minkowski functionals of the halo  $L(\omega)$. Here, we consider only the volume $\V$, the surface area $\Sur$ and the Euler-Poincaré characteristic $\chi$ (in dimension $d=2$). This restriction is due to statistical physics considerations as the stability of the energy is required. 
	\begin{definition}
		Let $\theta_1 \in \R$ and $\theta_2 \geq 0$. The energy of a finite configuration $\omega \in \Omega_f$ is given by 
		\begin{align*}
			H(\omega) = \begin{cases}
				\V(L(\omega)) + \theta_1 \Sur(L(\omega)) - \theta_2 \chi(L(\omega)) & \text{if } (d=2), \\
				\V(L(\omega)) + \theta_1 \Sur(L(\omega)) &  \text{if } (d\ge 3),
			\end{cases}
		\end{align*}
		where $\V(L(\omega))$ is the volume of $L(\omega)$ defined as the Lebesgue measure of $L(\omega)$, $\Sur(L(\omega))$ is the surface area of $L(\omega)$ defined as the $d-1$-dimensional Hausdorff measure of the boundary $\partial L(\omega)$ and $\chi(L(\omega))$ is the Euler-Poincaré characteristic of $L(\omega)$ defined as the difference between the number of connected components and the number of holes in $L(\omega)$ (in dimension $d=2$). 
	\end{definition}
	
	The energy is parametrized with two parameters $\theta_1$ and $\theta_2$. We discuss below why we impose $\theta_2$ to be non negative. Note that we do not introduce a third parameter in the front of the volume $\V(L(\omega))$ since it is indirectly given by the inverse temperature $\beta\ge 0$ in the Definition \ref{defGibbs} of Gibbs measures.
	
	With this choice of parameters the energy is stable which means that there exists a constant $C \geq 0$ such that for any finite configuration $\omega\in\Omega_f$,
	$$H(\omega) \geq - C N(\omega).$$
	The volume and the surface are clearly stable since the radii are uniformly bounded. The Euler-Poincaré characteristic is more delicate to study. In dimension 2, it is shown by Kendall et al. \Cite{Kendall} that for the union of N closed balls, the number of holes is bounded above by $2N-5$, and the number of connected components is clearly bounded by $N$. Therefore the Euler-Poincaré characteristic is stable for any parameter $\theta_2\in\R$. In higher dimension $d\geq 3$, for some configurations, the maximum number of holes is of order $N^2$ and thus the Euler-Poincaré characteristic is not stable if $\theta_2<0$. More generally, the stability of this statistic is not obvious even if $\theta_2$ is strictly positive. Therefore the existence of the infinite volume Gibbs point process is not well established. It is for this reason that we impose $\theta_2 =0$ in the case $d\ge3$.

	Let us turn to the definition of the local energy which provides the energy cost of introducing points in a domain given the configuration outside this domain. 
	\begin{definition}
		Let $\Delta \in \mathcal{B}_b(\R^d)$ and  $\omega \in \Omega_f$ be a finite configuration. We define the local energy of $\omega$ in $\Delta$ as
		\begin{equation*}
			H_\Delta (\omega) := H(\omega) - H(\omega_{\Delta^c}). 
		\end{equation*}
		From the additivity of Minkowski functionals, we observe a finite range property. Indeed we have that $H_\Delta ( \omega) = H_\Delta \big(\omega_{\Delta \oplus B(0,2R_1)}\big)$. Therefore, we can extend the  definition of the local energy to any configuration $\omega$ in $\Omega$ by
		\begin{equation}
			H_\Delta (\omega) := H_\Delta \left(\omega_{\Delta \oplus B(0,2R_1)}\right).
		\end{equation}
	\end{definition}

	\subsection{Gibbs measures}	
	
	Let $Q$ be a reference measure on $[R_0, R_1]$ for the distribution of the radii, and let $z$ be a non-negative real number  called the activity parameter. We denote by $\lambda$ the Lebesgue measure on $\R^d$ and by $\Pi^z$ the distribution of a Poisson point process on $E$ with intensity measure  $z \lambda \otimes Q$ \cite{last_penrose}. Similarly, for any $\Delta \in \mathcal{B}_b(\R^d)$, $\Pi_\Delta^z$ denotes the Poisson point process on $E_\Delta$ with intensity $z \lambda_\Delta \otimes Q$, where $\lambda_\Delta$ is the Lebesgue measure on $\Delta$. A point process $\mathbbm{X}$ is said to be stationary in space if, for any $u \in \R^d$, $\tau_u (\mathbbm{X}) \overset{d}{=} \mathbbm{X}$, where $\tau_u$ is the translation by the vector $u$. 
	
	\begin{definition}\label{defGibbs}
		A probability measure $P$ on $\Omega$ is a Gibbs measure for the Quermass interaction with parameters $\theta_1,\theta_2$, the activity $z >0$ and the inverse temperature $\beta \geq 0$ if $P$ is stationary in space and if for any $\Delta \in  \mathcal{B}_b(\R^d)$ and any bounded positive measurable function $f : \Omega \rightarrow \R$,
		\begin{equation}\label{eq_dlr}
			\int f(\omega) P(d\omega) = \int \int \frac{1}{Z_\Delta(\omega_{\Delta^c})} f(\omega_\Delta' \cup \omega_{\Delta^c}) e^{-\beta H_\Delta(\omega_\Delta' \cup \omega_{\Delta^c})} \Pi_\Delta^z(d\omega'_\Delta) P(d\omega)
		\end{equation} 
		where $Z_\Delta(\omega_{\Delta^c})$ is the partition function given the outer configuration $\omega_{\Delta^c}$
		\begin{equation*}
			Z_\Delta(\omega_{\Delta^c}) = \int e^{-\beta H_\Delta(\omega_\Delta' \cup \omega_{\Delta^c})} \Pi_\Delta^z(d\omega').
		\end{equation*}
	\end{definition}
	
	These equations are called DLR equations (for Dobrushin, Lanford and Ruelle). It is equivalent to the following conditional probability definition: For all $\Delta \in \mathcal{B}(\R^d)$ the distribution of $\omega_\Delta$ under $P$ given the outer configuration $\omega_{\Delta^c}$ is absolutely continuous with respect to $\Pi_\Delta^z$ with the following density 
	\begin{equation*}
		P(d\omega_\Delta' | \omega_{\Delta^c} ) = \frac{1}{Z_\Delta(\omega_{\Delta^c})} e^{-\beta H_\Delta(\omega_\Delta' \cup \omega_{\Delta^c})} \Pi_\Delta^z(d\omega'_\Delta).
	\end{equation*}
	For the collection of parameters $\theta_1, \theta_2, \beta$ and $z$, we denote by $\mathcal{G}(\theta_1, \theta_2, \beta, z)$ the set of all Gibbs measures for these parameters. The existence, the uniqueness or non-uniqueness (phase transition) of such Gibbs point processes are long-standing and challenging questions in statistical physics. In the present setting, since the interaction is finite range and stable, the existence is a direct application of Theorem 1 in \cite{MiniCours}. In other words, for any parameters $\theta_1\in\R, \theta_2\ge 0, \beta\ge 0$ and $z>0$, the set  $\mathcal{G}(\theta_1, \theta_2, \beta, z)$ is not empty. 
	
	\section{Results} \label{Section3}
	
	We say that a first-order phase transition occurs at $\theta_1,\theta_2, \beta $ and $z$  when the corresponding set of Gibbs measures $\mathcal{G}(\theta_1, \theta_2, \beta, z)$ contains at least two Gibbs measures $P^+,P^-$ with different intensities, i.e.,
	\begin{equation}
		\rho (P^+) := \E_P(N_{[0,1]^d}) > \rho(P^-).
	\end{equation}
	In particular, the set $\mathcal{G}(\theta_1, \theta_2, \beta, z)$ is not reduced to a singleton. This phenomenon is also called first order phase transition since the non-uniqueness of Gibbs measures is coupled with a discontinuity of the intensity. Other kind of phase transition are possible. Our main result states that such first-order phase transition occurs for a large range of parameters.

	We denote by $\theta_1^*$ the following constant $R_0 \frac{\V(B(0,1))}{\Sur(B(0,1))}$ = $\frac{R_0}{d}$. In dimension $d=2$, for any $\theta_1 > -\theta_1^*$ we define the constant $\theta_2^*(\theta_1) >0$ as in the equation \eqref{eq_theta_2_lim} below. At a first glance, the explicit value for $\theta_2^*(\theta_1)$ is not so important and its existence with a non-null value is already interesting in the next theorem.

	\begin{theorem}\label{thm_quermass_transition}
		Let $\theta_1,\theta_2$ be two parameters such that $\theta_1 > -\theta_1^* $ and  $0 \leq \theta_2 < \theta_2^*(\theta_1)$ (recall that $\theta_2=0$ if $d\ge 3$). Then there exists $\beta_c(\theta_1, \theta_2) > 0$ such that for all $\beta > \beta_c(\theta_1, \theta_2)$, there exists $ z_\beta^c > 0$ for which a first-order phase transition occurs: i.e. there exist $P^+,P^- \in \mathcal{G}(\theta_1, \theta_2, \beta, z_\beta^c)$ with $\rho(P^+) >  \rho(P^-)$. Moreover, there is $\alpha>0$ such that  $z_\beta^c$ satisfies $|z_\beta^c- \beta| = O(e^{-\alpha \beta})$.     
	\end{theorem}
	
	It is important to note that the first-order phase transition is proved only for a special value $z_\beta^c > 0$, the other parameters  $\theta_1,\theta_2, \beta$ being fixed. We are not able to prove that $z_\beta^c$ is the only one value for which the phase transition occurs; it is likely not true in general. But, as a corollary of the Pirogov-Sinaï-Zahradnik technology we used, in a neighbourhood of $z_\beta^c$ it is the case. This local uniqueness is typical for a first-order phase transition. Note that in the setting of Widom-Rowlinson model (i.e., $R_0=R_1$ and $\theta_1=\theta_2=0$), the critical activity $z_\beta^c$ is unique and equal to $\beta$ \cite{georgii_hagstrom, RuelleWR, WidomRowlinson}. Furthermore, we conjecture that the phase transition we observe is linked to a liquid-gas phase transition. Indeed, we believe that we do not have symmetry breaking for this model as nothing should induce rigidity in the behaviour. Our intuition, is that at low activity the model should roughly behave like a Poisson point process with intensity $ze^{-\beta H(\{0\})}$ and at high activity it behaves like a Poisson point process with intensity $z$. The reason is that at low activity, the halo is negligible, so points are primarily added in empty areas, with the cost of this action being $H(\{0\})$. Conversely, at high activity, the halo covers nearly the entire space, making the cost of adding a point essentially $0$. 
	
	The proof of Theorem \ref{thm_quermass_transition} relies on the study of the regularity of the so-called pressure $\psi$ defined by the following thermodynamic limit
	\begin{equation}
		\psi (z, \beta) : = \lim\limits_{n \rightarrow +\infty} \frac{1}{\beta |\Delta_n|} \ln(Z_{\Delta_n}),
	\end{equation}
	where  $(\Delta_n)$ is the sequence of the following boxes $[-n,n]^d$ and  $Z_{\Delta_n}$ is the partition function with free boundary condition (i.e.  $Z_{\Delta_n}(\emptyset)$). This limit always exists as a consequence of sub-additivity of the sequence $(\ln Z_{\Delta_n})_{n\in \N}$. As a corollary of the Pirogov-Sina\"i-Zahradnik technology we used, we obtain the non regularity of the pressure at the critical point.

	\begin{proposition}\label{prop_first_order_transition}
		Under the same assumptions as Theorem \ref{thm_quermass_transition}, for all $\theta_1,\theta_2, \beta$ such that $\theta_1 > -\theta_1^* $, $0 \leq \theta_2 < \theta_2^*(\theta_1)$  (recall that $\theta_2=0$ if $d\ge 3$) and $\beta > \beta_c(\theta_1, \theta_2)$ 
		\begin{align*}
			\frac{\partial \psi}{\partial z^+}(z_\beta^c, \beta) > \frac{\partial \psi}{\partial z^-}(z_\beta^c, \beta).
		\end{align*} 
		Furthermore, there exist two Gibbs measures $P^+, P^- \in \mathcal{G}(\theta_1, \theta_2, \beta, z_\beta^c)$  such that 
		\begin{equation}
			\rho(P^-) = z + z\beta\frac{\partial \psi}{\partial z^-}(z_\beta^c, \beta) \quad \text{and } \quad \rho(P^+) = z + z\beta \frac{\partial \psi}{\partial z^+}(z_\beta^c, \beta).
		\end{equation}
	\end{proposition}
	
	\begin{remark}
		In general, the Quermass interaction for $\omega \in \Omega_f$ is given by
		\begin{equation}
			H(\omega) = \sum_{k=0}^d \theta_{d-k} M_k^d(L(\omega))
		\end{equation}
		where $\theta_{d-k} \in \R$ and $M_k^d$ is the $k$-th Minkowski functionals. In particular, $M_d^d = \mathcal{V}$, $M_{d-1}^d = \mathcal{S}$ and $M_0^d = \chi$. We conjecture that the volume must contribute positively to the Hamiltonian ($\theta_0>0$) in order to observe a phase transition. It should be possible to extend our results and include additional Minkowski functionals. However, one must verify the stability assumption for such Hamiltonian and compare the contributions of $M_k^d$ and $\mathcal{V}$ in the contour. By doing so, it should be feasible to tune the parameters $\theta_k$ to satisfy a Peierls condition. 
	\end{remark}

	%It is another way to see the first order phase transition via the non differentiability of the pressure at the critical point. The right and left derivatives are related to the intensities of two corresponding Gibbs measures.
	
	\section{Proof}\label{Section4} 
	
	We start by providing the sketch of proof and an outline of the development. Our aim is to prove Proposition \ref{prop_first_order_transition}, as it is clear that this result implies Theorem \ref{thm_quermass_transition}. In section \ref{subsection4.1}, we perform a coarse graining of the Hamiltonian. We conduct a spatial decomposition of the energy function along the lines of tiles that pave $\R^d$, such that we have
	\begin{equation*}
		H = \sum_{i \in \Z^d} \boldmath{H}_i.
	\end{equation*}
	We can observe that for small tiles, the presence of at least one particle inside a tile $i \in \Z^d$ would imply that the spatial energy $\boldmath{H}_i = \delta^d$, where $\delta>0$ is suitably chosen and is the size of a tile. We call this property the saturation of the energy as no matter how many points we put in it this value remains constant. On the other hand, the spatial energy of an empty tile, when the other tiles around are empty, is equal to $0$. Therefore, when the status of the tiles are homogeneous, it becomes easy to compute the energy of a configuration. Hence the study of the behaviour in the mixing areas is necessary. In section \ref{subsection4.2}, we will properly define the contours, $\Gamma = \{ \gamma_1, \cdots, \gamma_n\}$, in order to match with the intuition of the interface areas between homogeneous regions. In section \ref{subsection4.3} we prove that the partition function with boundary condition associated to either state ( $\either = 1$ wired, $\either = 0$ free) can be written as a polymer development, i.e. 
	\begin{equation*}
		Z_\Lambda^\either = \sum_\Gamma \prod_{\gamma \in \Gamma} w_\gamma^\either
	\end{equation*}
	where $w_\gamma^\either$ are weights associated with each contours. In section \ref{subsection4.4}, for a large set of parameters $(\theta_1, \theta_2),$ we prove that the spatial energy in any contour $\gamma$ and any configuration $\omega$ achieving the contour verifies
	\begin{equation*}
		\boldmath{H}_\gamma (\omega) \geq |\gamma_1| \delta^d + \rho_0 |\gamma|,
	\end{equation*}
	where $\gamma_1 \subset \gamma$ corresponds to the tiles of the contours that contain a particle, and $\rho_0$ is a strictly positive constant. This type of result is in the same spirit of Peierls condition for contours on the lattice and comes from our spatial decomposition of the energy. If the size of the tiles are well tuned, some empty tiles would have the same spatial energy as a non-empty tile. We call this property the energy from vacuum. Using this inequality, we will prove, in section \ref{subsection4.5}, that for $\beta$ large enough and for $z$ in the vicinity of $\beta$ the weights are $\tau$-stable, i.e.
	\begin{equation*}
		w_\gamma^\either \leq e^{-\tau |\gamma|},
	\end{equation*}
	where $\tau >0$ depends on the value of $\beta$ and can be made arbitrarily large provided that $\beta$ is large. We prove this using truncated weights and pressure, as usual, with the PSZ theory. This exponential decay of the weights is needed to apply results on cluster expansion and write the partition function as follows
	\begin{equation*}
		Z_\Lambda^\either = e^{\psi |\Lambda| + \Delta_\Lambda},
	\end{equation*}
	where $\psi$ is the pressure and $\Delta_\Lambda$ is a perturbation term that is of the order of the surface. Finally in section \ref{subsection4.6}, we deduce the non-differentiability of the pressure for a critical activity and we deduce the first-order phase transition, as stated in Proposition \ref{prop_first_order_transition}.
	\newline
	
	\startcontents[sections]
	\printcontents[sections]{}{1}{}
	
	\subsection{Coarse Graining decomposition of the energy} \label{subsection4.1}
	
	Let $\delta > 0$ and for all integers $i \in \Z^d$, we call a tile $$T_i := \tau_{i\delta} \left( [ -\frac{\delta}{2}, \frac{\delta}{2}[^d\right),$$ where $\tau_{i\delta}$ is the translation by vector $i\delta$. We call a facet~$F$ any non-empty intersection between two closed tiles in  $(\overline{T}_i)_{i\in \Z}$. Clearly the dimension of a facet can be any integer between $0$ and $d$. We denote by $\mathcal{F}$ the set of all facets. The energy of a tile $H_i$ is given by
	\begin{equation*}
		\forall \omega \in \Omega_f, \quad H_i (\omega) = \V(L(\omega) \cap T_i) + \theta_1 \Sur_i(L(\omega)) - \theta_2 \chi_i (L(\omega)),
	\end{equation*}
	where, for any $A \subset \R^d$ a finite union of compact convex sets
	\begin{align*}
		\Sur_i (A) = \sum_{k = d-1}^d \sum_{\substack{F \in \mathcal{F} \\ \dim(F)=k}, F \cap T_i \neq \emptyset} (-1)^{d-k} \Sur(A \cap F) \\
		\chi_i(A) = \sum_{k = 0}^d \sum_{\substack{F \in \mathcal{F} \\ \dim(F)=k} , F \cap T_i \neq \emptyset} (-1)^{d-k} \chi(A \cap F).
	\end{align*} 
	Furthermore, we can observe that $S_i(\overline{T}_i) =0$ and $\chi_i(\overline{T}_i) = 0$. Therefore, for any finite configuration $\omega \in \Omega_f$ such that $T_i\subset L(\omega)$ we have $H_i(\omega) = \mathcal{V}(T_i) = \delta^d$. 
	
%	This is due to the fact that for a $d-1$-dimensional facet $F$, we have $\mathcal{S}(F) = 2 \lambda^{(d-1)}(F)$ where $\lambda^{(d-1)}$ is the $(d-1)$ Lebesgue measure.
	
	\begin{lemma}
		For every finite configuration $\omega \in \Omega_f$ 
		\begin{equation*}
			H(\omega) = \sum_{i \in \Z^d}  H_i(\omega).
		\end{equation*}
	\end{lemma}
	
	The proof follows directly from the additivity of the Minkowski functionals and, consequently, the inclusion-exclusion principle. In the following, for every $\Lambda \subset \Z^d$, we denote 
	\begin{equation*}
		H_\Lambda := \sum_{i \in \Lambda} H_i.
	\end{equation*}
	
	\subsection{Contours} \label{subsection4.2}
	
	Let us consider the lattice $\Z^d$ underlying the tiles $(T_i)_{i \in \Z^d}$, and the associated graph $G$ where two sites $i,j \in \Z^d$ are connected if $\norm{i-j}_\infty =1$. In this subsection, we consider a length $L>0$ used to define a contour. This distance will be specified later in the proof. We call the spin configuration the application 
	\begin{align*}
		\sigma :& \Omega_f \times \Z^d \rightarrow \{ 0, 1\} \\
		& (\omega, i) \mapsto \begin{cases}
			0 & \text{if } \omega_{T_i} =\emptyset \\
			1 & \text{otherwise}
		\end{cases}.
	\end{align*}
	In the following we use the notation $\either$ for either $0$ or $1$. 
	
	\begin{definition}
		Let $L>0$ and $\omega \in \Omega$, a site $i \in \Z^d$ is said to be $\either$-correct if for all sites $j$ such that $\norm{i-j} \leq L$, we have $\sigma(\omega, j) = \either$. A site $i$ is non-correct if it is neither $0$-correct nor $1$-correct. The set of all non-correct sites is denoted by $\overline{\Gamma}(\omega) $. We can partition $\overline{\Gamma}(\omega)$ into its maximal connected components with respect to the graph $G$. In the following, we denote these maximal connected components of $\overline{\Gamma}(\omega)$ by $\overline{\gamma}$.
	\end{definition}	
	Since we are considering only finite configurations, the number of connected components is finite and for any $\overline{\gamma}$, the complementary set has a finite amount of maximum connected components that we denote by $A$ and in particular we have only one unbounded connected component and we call it the exterior of $\overline{\gamma}$ that we denote by $ext(\overline{\gamma}) $.
	
	\begin{definition}
		Let $\Lambda \subset \Z^d$ and $L>0$, we define the exterior boundary~$\partial_{ext} \Lambda$ and the interior boundary~$\partial_{int} \Lambda$ of $\Lambda$ as
		\begin{align*}
			\partial_{ext} \Lambda = \{j \in \Lambda^c, d_2(j, \Lambda) \leq L \}  \\
			\partial_{int} \Lambda = \{i \in \Lambda, d_2(i, \Lambda^c) \leq L+1 \},
		\end{align*}
		where $d_2$ is the Euclidean distance in $\R^d$.		
	\end{definition}
	
	\begin{lemma}\label{lem_partial_type}
		Let $\omega \in \Omega_f$ be a finite configuration and $\overline{\gamma}$ be a maximal connected component of $\overline{\Gamma}(\omega)$. Let $A$ be a maximal connected component of $\overline{\gamma}^c$. Then, there is a unique $\either \in \{0, 1\}$ such that $\sigma(\omega, i) = \either$ for all $i \in \partial_{ext} A \cup \partial_{int}A$. This unique $\either$  is called the \emph{label} of $A$ and is denoted by $\Label(A)$. 
	\end{lemma}
	
	The proof of this lemma is classical and it corresponds to Lemma 7.23 in \cite{Velenik}. It relies on the fact that each set $\partial_{int}A$ and $\partial_{ext}A$ are connected and that the sites directly in contact with the contours are correct. Therefore there can be only one spin $\either \in \{0,1\}$ otherwise we would have two correct sites of opposite spin directly connected. 
	
	\begin{definition}
		Let $\omega \in \Omega_f$, we call a contour~$\gamma$ the pair $ (\overline{\gamma}, (\either_i)_{i \in \overline{\gamma}})$ where for all sites~$i \in \overline{\gamma}$, $\sigma(\omega, i) = \either_i$. We denote by $\Gamma(\omega)$ the set of all contours that appear with the configuration $\omega$. 
	\end{definition}
	
	Furthermore for a contour~$\gamma = (\overline{\gamma}, (\either_j)_{j \in \overline{\gamma}})$ we call the \emph{type} of  $\gamma$ the label of $ext(\overline{\gamma})$, $\Type(\gamma) := \Label(ext(\overline{\gamma}))$. And we call the interiors of a contour $\gamma$ the sets 
	\begin{equation*}
		\Int_\either \gamma = \bigcup_{\substack{A \neq ext(\overline{\gamma}) \\  \Label(A) = \either}} A  \quad  \text{ and } \quad \Int \gamma = \Int_0 \gamma \cup \Int_1 \gamma.
	\end{equation*}
	Let $\omega \in \Omega_f$ be a finite configuration, a contour $\gamma \in \Gamma(\omega)$ is said to be external when for any other contour $\gamma' \in \Gamma(\omega)$,  $\overline{\gamma} \subset ext(\overline{\gamma}')$. We denote by $\Gamma_{ext}$ the subset of $\Gamma$ comprised only of external contours. 
	
	\medskip
	
	Until now we have only considered collection of contours that can be achieved by a finite configuration of points. But classically in the Pirogov-Sinaï-Zahradnik theory we need to introduce abstract collection of contours which are not achievable by any configuration. This is due to the cluster expansion development of the partition function using geometrically compatible collection of contours.
	
	\begin{definition} An abstract set of contours
		is a set of contours $\{ \gamma_i = (\overline{\gamma}_i, (\either_j)_{j \in \overline{\gamma_i}}) , i \in I  \subset \N^*\}$ for which  each contour $\gamma_i$ is achievable for some configuration $\omega_i$. We do not assume the global achievability.  We denote by $\Gamma$ such set of contours. Moreover this set $\Gamma$ is called geometrically compatible if for all $\{i,j\} \subset I$,  $ d_\infty(\overline{\gamma}_i, \overline{\gamma}_j)>1$ and they all have the same type. Let $\Lambda \in \Z^d$ we denote by $\mathcal{C}^\either(\Lambda)$ the collection of geometrically compatible sets of contours of the type $\either$ such that $d_\infty(\overline{\gamma}_i, \Lambda^c)>1$. 
	\end{definition}
	
	\begin{figure}[H]
		
		\subfloat[] {
			\begin{tikzpicture}[scale=0.9]\label{fig_a}
				% Définir la taille de la grille
				\def\grillesize{21} % Nombre de cases par côté
				\def\cellsize{0.4cm}   % Taille d'une case (en cm)
				
				\foreach \x/\y in {3/3, 3/4, 3/5, 4/6, 4/7, 3/8,
					5/9, 5/10, 5/11, 6/12, 5/13, 4/14, 4/15, 3/16, 3/17, 5/18, 4/18, 4/17, 4/16, 5/15, 5/14, 4/13, 5/12, 4/11, 4/10, 4/9, 4/8, 5/7, 5/6, 4/5, 4/4, 4/3, 4/2, 5/2, 5/3, 6/2, 6/3, 7/3, 7/4, 8/3, 8/4, 9/3, 9/4, 10/2, 10/3, 11/2, 12/2, 11/3, 12/3, 12/4, 13/3, 13/4, 14/3, 14/4, 15/3, 15/4, 16/2, 16/3, 17/2, 17/3, 17/4, 18/3, 18/4, 18/5, 17/5, 18/6, 17/6, 18/7, 17/7, 16/8, 17/8, 16/9, 17/9, 16/10, 17/10, 17/11, 18/11, 17/12, 18/12, 17/13, 18/13, 16/14, 17/14, 16/15, 17/15, 17/16, 18/16, 17/17, 18/17, 5/17, 6/17, 6/18, 7/16, 7/17, 8/16, 8/17, 9/16, 9/17, 10/17, 10/18, 11/17, 11/18, 12/17, 12/18, 13/16, 13/17, 14/16, 14/17, 15/17, 15/18, 16/16, 16/17, 11/9, 11/10, 11/11, 12/10, 12/11, 10/10}{
					\fill[gray!50] 
					({\x*\cellsize}, {\y*\cellsize})
					rectangle 
					({(\x+1)*\cellsize}, {(\y+1)*\cellsize});
				}
				
				\foreach \x/\y in {2/2, 3/2, 3/1, 4/1, 5/1, 6/1, 7/1, 7/2, 8/2, 9/2, 9/1, 10/1, 11/1, 12/1, 13/1, 13/2, 14/2, 15/2, 15/1, 16/1, 17/1, 18/1, 18/2, 19/2, 19/3, 19/4, 19/5, 19/6, 19/7, 19/8, 18/8, 18/9, 18/10, 19/10, 19/11, 19/12, 19/13, 19/14, 18/14, 18/15, 19/15, 19/16, 19/17, 19/18, 18/18, 17/18, 16/18, 16/19, 15/19, 14/19, 14/18, 13/18, 13/19, 12/19, 11/19, 10/19, 9/19, 9/18, 8/18, 7/18, 7/19, 6/19, 5/19, 4/19, 3/19, 3/18, 2/18, 2/17, 2/16, 2/15, 3/15, 3/14, 3/13, 3/12, 4/12, 3/11, 3/10, 3/9, 2/9, 2/8, 2/7, 3/7, 3/6, 2/6, 2/5, 2/4, 2/3 } {
					\fill[blue!50] 
					({\x*\cellsize}, {\y*\cellsize})
					rectangle 
					({(\x+1)*\cellsize}, {(\y+1)*\cellsize});
				}
				
				\foreach \x/\y in {6/4, 5/4, 5/5, 6/5, 6/6, 6/7, 6/8, 5/8, 6/9, 6/10, 6/11, 7/11, 7/12, 7/13, 6/13, 6/14, 6/15, 6/16, 5/16, 7/15, 8/15, 9/15, 10/15, 10/16, 11/16, 12/16, 12/15, 13/15, 14/15, 15/15, 15/16, 15/14, 15/13, 16/13, 16/12, 16/11, 15/11, 15/10, 15/9, 15/8, 15/7, 16/7, 16/6, 16/5, 16/4, 15/5, 14/5, 13/5, 12/5, 11/5, 11/4, 10/4, 10/5, 9/5, 8/5, 7/5} {
					\fill[red!50] 
					({\x*\cellsize}, {\y*\cellsize})
					rectangle 
					({(\x+1)*\cellsize}, {(\y+1)*\cellsize});
				}
				
				\foreach \x/\y in {10/9, 9/9, 9/10, 9/11, 10/11, 10/12, 11/12, 12/12, 13/12, 13/11, 13/10, 13/9, 12/9, 12/8, 11/8, 10/8} {
					\fill[red!50] 
					({\x*\cellsize}, {\y*\cellsize})
					rectangle 
					({(\x+1)*\cellsize}, {(\y+1)*\cellsize});
				}
				% Dessiner la grille
				\foreach \x in {0,1,...,\grillesize} {
					% Lignes horizontales
					\draw[black, thin] (0, \x*\cellsize) -- (\grillesize*\cellsize, \x*\cellsize);
					% Lignes verticales
					\draw[black, thin] (\x*\cellsize, 0) -- (\x*\cellsize, \grillesize*\cellsize);
				}
				\node[ font=\bfseries\Large] at (11.5*\cellsize,10.5*\cellsize) {$\gamma_2$};
				
				\node[font=\bfseries\Large] at (11.5*\cellsize,17.5*\cellsize) {$\gamma_1$};
				
				\node[font=\bfseries\Large] at (11.5*\cellsize,14.5*\cellsize) {$A$};
			\end{tikzpicture}
		}
		\quad
		\subfloat[] {
			\begin{tikzpicture}[scale=0.9]\label{fig_b}
				% Définir la taille de la grille
				\def\grillesize{21} % Nombre de cases par côté
				\def\cellsize{0.4cm}   % Taille d'une case (en cm)
				
				\foreach \x/\y in { 3/3, 3/4, 3/5, 4/6, 4/7, 3/8, 5/9, 5/10, 5/11, 6/12, 5/13, 4/14, 4/15, 3/16, 3/17, 5/18, 4/18, 4/17, 4/16, 5/15, 5/14, 4/13, 5/12, 4/11, 4/10, 4/9, 4/8, 5/7, 5/6, 4/5, 4/4, 4/3, 4/2, 5/2, 5/3, 6/2, 6/3, 7/3, 7/4, 8/3, 8/4, 9/3, 9/4, 10/2, 10/3, 11/2, 11/3, 12/2, 12/3, 12/4, 13/3, 13/4, 14/3, 14/4, 15/3, 15/4, 16/2, 16/3, 17/2, 17/3, 17/4, 18/3, 18/4, 18/5, 17/5, 18/6, 17/6, 18/7, 17/7, 16/8, 17/8, 16/9, 17/9, 16/10, 17/10, 17/11, 18/11, 17/12, 18/12, 17/13, 18/13, 16/14, 17/14, 16/15, 17/15, 17/16, 18/16, 17/17, 18/17, 5/17, 6/17, 6/18, 7/16, 7/17, 8/16, 8/17, 9/16, 9/17, 10/17, 10/18, 11/17, 11/18, 12/17, 12/18, 13/16, 13/17, 14/16, 14/17, 15/17, 15/18, 16/16, 16/17, 11/9, 11/10, 11/11, 12/10, 12/11, 10/10} {
					\fill[gray!50] 
					({\x*\cellsize}, {\y*\cellsize})
					rectangle 
					({(\x+1)*\cellsize}, {(\y+1)*\cellsize});
				}
				
				\foreach \x/\y in {2/2, 3/2, 3/1, 4/1, 5/1, 6/1, 7/1, 7/2, 8/2, 9/2, 9/1, 10/1, 11/1, 12/1, 13/1, 13/2, 14/2, 15/2, 15/1, 16/1, 17/1, 18/1, 18/2, 19/2, 19/3, 19/4, 19/5, 19/6, 19/7, 19/8, 18/8, 18/9, 18/10, 19/10, 19/11, 19/12, 19/13, 19/14, 18/14, 18/15, 19/15, 19/16, 19/17, 19/18, 18/18, 17/18, 16/18, 16/19, 15/19, 14/19, 14/18, 13/18, 13/19, 12/19, 11/19, 10/19, 9/19, 9/18, 8/18, 7/18, 7/19, 6/19, 5/19, 4/19, 3/19, 3/18, 2/18, 2/17, 2/16, 2/15, 3/15, 3/14, 3/13, 3/12, 4/12, 3/11, 3/10, 3/9, 2/9, 2/8, 2/7, 3/7, 3/6, 2/6, 2/5, 2/4, 2/3 } {
					\fill[blue!50]
					({\x*\cellsize}, {\y*\cellsize})
					rectangle 
					({(\x+1)*\cellsize}, {(\y+1)*\cellsize});
				}
				
				\foreach \x/\y in {6/4, 5/4, 5/5, 6/5, 6/6, 6/7, 6/8, 5/8, 6/9, 6/10, 6/11, 7/11, 7/12, 7/13, 6/13, 6/14, 6/15, 6/16, 5/16, 7/15, 8/15, 9/15, 10/15, 10/16, 11/16, 12/16, 12/15, 13/15, 14/15, 15/15, 15/16, 15/14, 15/13, 16/13, 16/12, 16/11, 15/11, 15/10, 15/9, 15/8, 15/7, 16/7, 16/6, 16/5, 16/4, 15/5, 14/5, 13/5, 12/5, 11/4, 10/4, 11/5, 10/5, 9/5, 8/5, 7/5} {
					\fill[red!50] 
					({\x*\cellsize}, {\y*\cellsize})
					rectangle 
					({(\x+1)*\cellsize}, {(\y+1)*\cellsize});
				}
				
				\foreach \x/\y in {10/9, 9/9, 9/10, 9/11, 10/11, 10/12, 11/12, 12/12, 13/12, 13/11, 13/10, 13/9, 12/9, 12/8, 11/8, 10/8} {
					\fill[blue!50] 
					({\x*\cellsize}, {\y*\cellsize})
					rectangle 
					({(\x+1)*\cellsize}, {(\y+1)*\cellsize});
				}
				% Dessiner la grille
				\foreach \x in {0,1,...,\grillesize} {
					% Lignes horizontales
					\draw[black, thin] (0, \x*\cellsize) -- (\grillesize*\cellsize, \x*\cellsize);
					% Lignes verticales
					\draw[black, thin] (\x*\cellsize, 0) -- (\x*\cellsize, \grillesize*\cellsize);
				}
				\node[ font=\bfseries\Large] at (11.5*\cellsize,10.5*\cellsize) {$\gamma_2$};
				
				\node[font=\bfseries\Large] at (11.5*\cellsize,17.5*\cellsize) {$\gamma_1$};
				
				\node[font=\bfseries\Large] at (11.5*\cellsize,14.5*\cellsize) {$A$};
				
			\end{tikzpicture} 
		}
		\caption{The contour corresponds to the grey areas, while the blue and red squares represent the tiles at the boundary of the contour where the spins are $\either$ and $1-\either$, respectively. In Figure \ref{fig_a}, the contour $\Gamma =\{\gamma_1, \gamma_2\}$ is achievable by some configuration $\omega$ because the label of $A$ matches for both $\gamma_1$ and $\gamma_2$. In contrast, in Figure \ref{fig_b}, the contour $\Gamma = \{\gamma_1, \gamma_2\}$ is not globally achievable by any configuration. In this case, $\Gamma \in \mathcal{C}^\either(\Lambda)$, where the types of $\gamma_1$ and $\gamma_2$ are the same, but the labels of $A$ are mismatched. }
	\end{figure}
	
	We allow the set $\Gamma = \{ (\emptyset, \emptyset)\}$ to belong to the collection $\mathcal{C}^\either(\Lambda)$ for any $\Lambda$ which corresponds to the event where not a single contour appears in $\Lambda$. There are several interesting sub-collections of $\mathcal{C}^\either(\Lambda)$, one of them being the collection of sets such that all contours are external. 
	
	\begin{definition}
		Let $\Lambda \subset \Z^d$ finite, we denote by $\mathcal{C}_{ext}^\either(\Lambda) \subset \mathcal{C}^\either(\Lambda)$ the sub-collection of sets $\Gamma$ where any contours $\gamma \in \Gamma$ is external. 
	\end{definition}
	
	In a way, in the collection $\mathcal{C}_{ext}^\either(\Lambda)$ we are considering sets of contours where we have only one layer. In general, if we take a geometrically compatible abstract sets of contours $\Gamma$, a particular contour in this set can be encapsulated in the interior of another creating layers upon layers of contours. One method of exploration of the contours is by proceeding from the external layer and peel each layer to discover the other contours hidden under. Another sub-collection of $\mathcal{C}^\either(\Lambda)$ is the collection of sets such that for all contours the size of the interior is bounded. 
	
	\begin{definition}
		A contour $\gamma$ is of the class $k \in \N$ when $|\Int \gamma | =k$. Let $n \in \N$ and $\Lambda \subset \Z^d$ finite, we denote by $\mathcal{C}_n^\either(\Lambda) \subset \mathcal{C}^\either(\Lambda)$ the collection of contours $\Gamma$ such that $\forall \gamma \in \Gamma$, $\gamma$ is of the class $k\leq n$. 
	\end{definition}

	\subsection{Partition function and boundary condition} \label{subsection4.3}
	
	 For any $\Lambda \subset \Z^d$, let $\widehat{\Lambda} := \bigcup_{i \in \Lambda} T_i$. We consider two Quermass point processes with a free boundary and a wired boundary conditions. The probability measures~$P_{\Lambda}^\either$ associated with each boundary condition are given by 
	\begin{equation}\label{def_mesure_gibbs_condition_bord}
		P_\Lambda^\either(d\omega) = \frac{1}{Z_\Lambda^\either} e^{-\beta H_\Lambda(\omega)} \mathbbm{1}_{\{\forall i \in \partial_{int} \Lambda, \sigma(\omega, i) = \either\}} \Pi_{\widehat{\Lambda}}^z(d\omega)
	\end{equation}
	where 
	\begin{equation}
		Z_\Lambda^\either = \int e^{-\beta H_\Lambda(\omega)} \mathbbm{1}_{\{\forall i \in \partial_{int} \Lambda, \sigma(\omega, i) = \either\}} \Pi_{\widehat{\Lambda}}^z(d\omega). 
	\end{equation}
	We prove that the two infinite Gibbs measures obtained by taking the thermodynamic limit for each boundary condition yield different intensities. First, we have a standard lemma which ensures that the pressure does not depend on the boundary conditions.
	
	\begin{lemma}\label{lem_pression_free_boundary}
		For $\delta \le \nicefrac{R_0}{2\sqrt{d}}$, $L\ge \nicefrac{2R_1}{\delta}$ and any $\either \in \{0,1\}$ we have $\psi = \psi^\either$ where $\psi^\either := \underset{n \to \infty}{\lim} \frac{\ln Z_{\Lambda_n}^\either}{\beta |\Lambda_n| \delta^d}$ and $\Lambda_n = [\![-n,n]\!]^d$.
	\end{lemma}
	
	\begin{proof}
		For any configuration $\omega \in \Omega_f$ such that for all sites $i \in \partial_{int} \Lambda_n, \sigma(\omega,i) = 1$ and $\omega_{\Lambda_n^c} = \emptyset$ we know that no holes are created by the halo outside of $\Lambda_n$ and therefore
		\begin{equation*}
			H_{\partial_{ext} \Lambda_n} (\omega) \leq \begin{cases}
				|\partial_{ext} \Lambda_n | \delta^d & \text{ when } \theta_1 \leq 0 \\
				|\partial_{ext} \Lambda_n | \delta^d + \theta_1 \mathcal{S}_{\partial_{ext}\Lambda_n}(L(\omega))& \text{ when } \theta_1 >0
			\end{cases}.
		\end{equation*}
		The boundary of the halo $L(\omega)$ outside $\Lambda_n$, appearing in the computation of $\mathcal{S}_{\partial_{ext}\Lambda_n}(L(\omega))$, is the union of spherical caps built via some marked points $(x_1,r_1), \dots, (x_m, r_m) \in \omega$. We denote by $\alpha_i\in[0,1]$ the ratio of the surface of the $i$th spherical cap with respect to the total surface of its sphere. We have
		\begin{align*}
			\mathcal{S}_{\partial_{ext} \Lambda_n}(L(\omega)) &= \sum_{i = 1}^m \alpha_i \mathcal{S}(B(x_i, r_i)) \\
			& = \sum_{i = 1}^m \alpha_i \frac{\mathcal{S}(B(x_i, r_i))}{\mathcal{V}(B(x_i, r_i))} \mathcal{V}(B(x_i, r_i)) \\
			& \leq \frac{\mathcal{S}(B(0, 1))}{R_0\mathcal{V}(B(0, 1))} \sum_{i = 1}^m \alpha_i \mathcal{V}(B(x_i, r_i)) \\
			& \leq \frac{\mathcal{S}(B(0, 1))}{R_0\mathcal{V}(B(0, 1))} |\partial_{int} \Lambda_n \cup \partial_{ext} \Lambda_n |\delta^d.
		\end{align*}
		As a consequence there exists $c>0$ such that we have
		\begin{align*}
			Z_{\widehat{\Lambda}_n} & \geq \int e^{-\beta H_{\partial_{ext} \Lambda_n} (\omega) } e^{-\beta H_{\Lambda_n} (\omega) } \mathbbm{1}_{\{\forall i \in \partial_{int}\Lambda_n, \sigma(\omega, i)= 1\}} \Pi_{\widehat{\Lambda}_n}^z(d\omega) \\
			& \geq e^{-c |\partial_{ext} \Lambda_n \cup \partial_{int} \Lambda_n| \delta^d } Z_{\Lambda_n}^{(1)}.
		\end{align*}
		Therefore we have $\psi \geq \psi^{(1)}$. Let us consider the following event 
		\begin{equation*}
			E_n = \{ \omega \in \Omega_f, \forall i \in \Lambda_n, L < d_2(i, \Lambda_n^c) \leq 2L, \sigma(\omega,i) = 0\}.
		\end{equation*}
		Since $\delta \le \nicefrac{R_0}{2 \sqrt{d}}$ and $L \ge \frac{2R_1}{\delta}$ we have
		\begin{align*}
			Z_{\Lambda_n}^{(1)} \geq \int e^{-\beta (|\partial_{int} \Lambda_n| + H_{\partial_{int} \Lambda_{n-L}}(\omega_{\widehat{\partial_{int} \Lambda_n}}))} \mathbbm{1}_{\{\forall i \in \partial_{int}\Lambda_n, \sigma(\omega, i)= 1\}} e^{-\beta H_{\Lambda_{n-L}}(\omega_{\widehat{\Lambda}_{n-L}})} \mathbbm{1}_{E_n} \Pi_{\widehat{\Lambda}_n}^z(d\omega).
		\end{align*}
		Similarly as in the previous case we can find $c>0$ such that 
		\begin{equation*}
			H_{\partial_{int} \Lambda_{n-L}}(\omega_{\widehat{\partial_{int} \Lambda_n}}) \leq c |\partial_{int} \Lambda_n \cup \partial_{ext} \Lambda_n| \delta^d 
		\end{equation*}
		and therefore we get
		\begin{equation*}
			Z_{\Lambda_n}^{(1)} \geq g_1^{|\partial_{int} \Lambda_n|} e^{-\beta c |\partial_{int} \Lambda_n \cup \partial_{ext} \Lambda_n| \delta^d } Z_{\Lambda_{n-L}}^{(0)}. 
		\end{equation*}
		This implies that $\psi^{(1)} \geq \psi^{(0)}$. Finally we simply have that 
		\begin{equation*}
			Z_{\Lambda_n}^{(0)} \geq g_0^{|\partial_{int} \Lambda_n|} Z_{\Lambda_{n-L}}.
		\end{equation*}
		Therefore $\psi^{(0)} \geq \psi$ and this finishes the proof of the lemma. 
	\end{proof}

	We have also a crucial proposition which provides a representation of the partition function as a polymer development.   
	
	\begin{proposition}[Polymer development]
		Let $R_0 \geq 2\delta \sqrt{d}>0$  and $\delta L>2R_1$, for all $\Lambda \subset \Z^d$ finite we have 
		\begin{equation*}
			\Phi_\Lambda^\either := g_\either^{-|\Lambda|} Z_\Lambda^\either = \sum_{\Gamma \in  \mathcal{C}^\either(\Lambda)}  \prod_{\gamma \in \Gamma} w_{\gamma}^\either
		\end{equation*}
		where
		\begin{itemize}
			\item $g_0 = e^{-z\delta^d}$ and $g_1 = e^{-\beta \delta^d}(1-e^{-z\delta^d})$ \\
			\item $w_\gamma^\either$ is called the weight of the contour $\gamma$ and if we denote $\#* := 1 - \#$, it is given by the following formula $$w_{\gamma}^\either = g_\either^{-|\overline{\gamma}|} I_{\gamma} \frac{Z_{\Int_{\either^*} \gamma}^{\either^*}}{Z_{\Int_{\either^*} \gamma}^{\either}},$$ \\
			\item $I_{\gamma} = \int e^{-\beta H_{\overline{\gamma}} (\omega)} \mathbbm{1}_{\left(\forall i \in \overline{\gamma}, \sigma(\omega,i)=\sigma_i \right)} \Pi_{\widehat{\gamma}}^z(d\omega)$.
		\end{itemize}
	\end{proposition}
	
	\begin{proof}
		We follow a similar development done in Chapter 7 in \cite{Velenik} with an adaptation to the setting of our model where the main difference is that the states of sites are random and have to be integrated under the Poisson measure. We can decompose the partition function $Z_\Lambda^\either$ according to the external contours $\mathcal{C}_{ext}^\either(\Lambda)$ and we have 
		\begin{align*}
			Z_\Lambda^\either = \sum_{\Gamma \in \mathcal{C}_{ext}^\either(\Lambda)} \int e^{-\beta H_{\Lambda}(\omega)} \mathbbm{1}_{\{\forall i \in \partial_{int}\Lambda, \sigma(\omega, i)= \either\}} \mathbbm{1}_{\{\Gamma_{ext}(\omega) = \Gamma\}} \Pi_{\widehat{\Lambda}}^z(d\omega).
		\end{align*}
		For any $\Gamma \in \mathcal{C}_{ext}^\either(\Lambda)$ we can do a partition of $\Lambda$ in the following way
		\begin{equation*}
			\Lambda =  \Lambda_{ext} \bigcup_{\substack{\gamma \in \Gamma}} \left( \gamma \cup \Int_{0}\gamma \cup \Int_{1}\gamma \right) 
		\end{equation*}
		where $\Lambda_{ext} = \bigcap_{{\gamma} \in \Gamma} ext(\overline{\gamma}) \cap \Lambda$. For any finite configuration $\omega \in \Omega_f$ such that $\Gamma_{ext}(\omega) = \Gamma$,  we have 
		\begin{align}
			&H_{\Lambda_{ext}} (\omega) = H_{\Lambda_{ext}} ( \omega_{{\widehat{\Lambda}}_{ext}}) = \mathcal{V}({{\widehat{\Lambda}}_{ext}}) \mathbbm{1}_{(\either=1)} \label{energy_local_gamma_1}\\
			&H_{\overline{\gamma}} (\omega) = H_{\overline{\gamma}} (\omega_{\widehat{\gamma}}) \label{energy_local_gamma_2}\\
			&H_{\Int_\either \gamma}(\omega) = H_{\Int_\either \gamma}(\omega_{\widehat{\Int_\either \gamma}}). \label{energy_local_gamma_3}
		\end{align}
		By construction, we know that for all $i \in \Lambda_{ext}, \sigma(\omega, i) = \either$ and $i$ is $\either$-correct. First case $\either = 1$, since $R_0 \geq \delta \sqrt{d}$ if $\omega_{T_i} \neq \emptyset$ implies that $H_i(\omega) = H_i(\omega_{T_i}) = \mathcal{V}(T_i)$. Second case $\either = 0$, since the sites $i$ in $\Lambda_{ext}$ are $0$-correct and that $\delta L > 2 R_1$ then $H_i(\omega) = H_i(\omega_{T_i}) = 0$.
		
		Furthermore, we know that a tile adjacent to an external contour $\gamma$ is correct.  For $A= \Int_0 \gamma $  or $A = \Lambda_{ext}$ when $\either = 0$ , because of the previous fact we know that $d_2(\omega_{\widehat{A}}, \widehat{\gamma}) > 2R_1$ and $d_2(\omega_{\widehat{\gamma}}, \widehat{A}) > 2R_1$. Therefore we have $L(\omega_{\widehat{\gamma}}) \cap \widehat{A} = \emptyset$ and $L(\omega_{\widehat{A}}) \cap \widehat{\gamma}= \emptyset$. 
		And if we consider $ B = \Int_1 \gamma $ or $B = \Lambda_{ext}$ when $\either=1$, the energy of the tiles in $\partial_{int}B$ and $\partial_{ext}B$ are determined since we have a point in those tiles and that $R_0  \geq \delta \sqrt{d}$. Furthermore,  since $\delta  L > 2R_1$  what is happening in the tiles $\overline{\gamma} \backslash \partial_{ext}B$ doesn't have any consequence on $B$ and vice versa what is happening in $B \backslash \partial_{int}B$ doesn't affect the tiles in $\overline{\gamma}$. Using these observations on the way the energy behaves according the contours and using the independence of Poisson point process in disjoint areas we have
		\begin{align}
			Z_\Lambda^\either =& \sum_{\Gamma \in \mathcal{C}_{ext}^\either(\Lambda)} \left( \int e^{-\beta H_{\Lambda_{ext}}(\omega)} \mathbbm{1}_{\{ \forall i \in \Lambda_{ext}, \sigma(\omega,i)=\either \}} \Pi_{\widehat{\Lambda_{ext}}}^z(d\omega) \right) \nonumber \\
			&\times \prod_{\gamma \in \Gamma} \left( \int e^{-\beta H_{\overline{\gamma}} (\omega)} \mathbbm{1}_{\{ \forall i \in \overline{\gamma}, \sigma(\omega, i) = \sigma_i \}} \Pi_{\widehat{\overline{\gamma}}}^z(d\omega) \right) \nonumber \\
			& \times \prod_{\either^*} \left( \int e^{-\beta H_{\Int_{\either^*} \gamma}(\omega)} \mathbbm{1}_{\{ \forall i \in \partial_{int} \Int_{\either^*} \gamma, \sigma(\omega, i) = \either^* \}}	 \Pi_{\widehat{\Int_{\either^*} \gamma}}^z(d\omega) \right). \label{eq_decomposition_product_integrals}
		\end{align}
		In this development we recognize the partition function $Z_{\Int_{\either^*} \gamma}^{\either^*}$ and the term $I_\gamma$. At the end, we have by independence of Poisson point process in disjoint areas
		\begin{align*}
			\int e^{-\beta H_{\Lambda_{ext}}(\omega)} \mathbbm{1}_{\{ \forall i \in \Lambda_{ext}, \sigma(\omega,i)=\either \}} \Pi_{\widehat{\Lambda_{ext}}}^z(d\omega) = \left( \int e^{-\beta H_0(\omega)} \mathbbm{1}_{\{ \sigma(\omega,0)=\either \}} \Pi_{T_0}^z(d\omega)\right)^{|\Lambda_{ext}|}.
		\end{align*}
		According to the value $\either \in \{ 0,1\}$ we have 
		\begin{align*}
			&\int e^{-\beta H_0(\omega)} \mathbbm{1}_{\{ \sigma(\omega,0)=1 \}} \Pi_{T_0}^z(d\omega) = e^{-\beta \delta^d}(1-e^{-z\delta^d}) = g_1\\
			&\int e^{-\beta H_0(\omega)} \mathbbm{1}_{\{ \sigma(\omega,0)=0 \}} \Pi_{T_0}^z(d\omega) = e^{-z \delta^d} = g_0.
		\end{align*}
		In summary, we have the following 
		\begin{align*}
			Z_\Lambda^\either &= \sum_{\Gamma \in \mathcal{C}_{ext}^\either(\Lambda)} g_{\either}^{|\Lambda_{ext}|} \prod_{\gamma \in \Gamma} I_{\gamma} Z_{\Int_\either \gamma}^{\either} Z_{\Int_{\either^*} \gamma}^{\either^*} \\
			&= \sum_{\Gamma \in \mathcal{C}_{ext}^\either(\Lambda)} g_{\either}^{|\Lambda_{ext}|} \prod_{\gamma \in \Gamma} I_{\gamma} \frac{Z_{\Int_{\either^*} \gamma}^{\either^*}}{Z_{\Int_{\either^*} \gamma}^{\either}} Z_{\Int_\either \gamma}^{\either} Z_{\Int_{\either^*} \gamma}^{\either} \\
			&= \sum_{\Gamma \in \mathcal{C}_{ext}^\either(\Lambda)} g_{\either}^{|\Lambda_{ext}|} \prod_{\gamma \in \Gamma} g_{\either}^{|\overline{\gamma}|} w_\gamma^\either Z_{\Int_\either \gamma}^{\either} Z_{\Int_{\either^*} \gamma}^{\either}.
		\end{align*}
		From the properties our energy, we know that $0 <Z_{\Int_{\either^*} \gamma}^{\either} < +\infty$ and thus the quotient introduced above is allowed. In regards of the quantity $\Phi_\Lambda^\either$ we have
		\begin{align*}
			\Phi_\Lambda^\either = \sum_{\Gamma \in \mathcal{C}_{ext}^\either(\Lambda)}  \prod_{\gamma \in \Gamma}  w_\gamma^\either \Phi_{\Int_\either \gamma}^{\either} \Phi_{\Int_{\either^*} \gamma}^{\either}.
		\end{align*}
		We can continue to iterate the same computation for $\Phi_{\Int_\either \gamma}^\either$ and for $\Phi_{\Int_{\either^*} \gamma}^\either$ until we empty the interior of the contours and thus we have 
		\begin{align*}
			\Phi_\Lambda^\either = \sum_{\Gamma \in \mathcal{C}^\either(\Lambda)}  \prod_{\gamma \in \Gamma}  w_\gamma^\either.
		\end{align*}
	\end{proof}
	
	\subsection{Energy and Peierls condition} \label{subsection4.4}
	
	The weights of any contours $\gamma$ is a difficult object to evaluate and our goal is to have a good exponential bound with respect the volume of the contour. Recall the expression of the weight
	\begin{align*}
		w_{\gamma}^\either = g_\either^{-|\overline{\gamma}|} I_{\gamma} \frac{Z_{\Int_{\either^*} \gamma}^{\either^*}}{Z_{\Int_{\either^*} \gamma}^{\either}}.
	\end{align*}	
	
	\begin{definition}
		Let $\tau > 0$, the weight of a contour $\gamma$ is said to be $\tau$-stable if
		\begin{equation*}
			w_{\gamma}^\either \leq e^{-\tau |\overline{\gamma}|}.
		\end{equation*}
	\end{definition}
	
	The ratio of the partition functions is the most difficult part to handle at this stage and will be further developed in the next section. For the moment, we want to control the quantity $I_{\gamma}$ where $\ln{I_{\gamma}}$ can be interpreted as the mean energy of the contour. The aim is to display some sort of Peierls condition on the mean energy of the contour and afterwards to prove the $\tau$-stability of the weights. 	
	
	\begin{proposition}\label{prop_peierls}
		For any $\theta_1 > - \theta_1^*:= -  R_0 \frac{\mathcal{V}(B(0,1))}{\mathcal{S}(B(0,1))}$ (and, in dimension $d=2$, any $\theta_2$ such that for $0 \leq \theta_2 < \theta_2^*(\theta_1)$, where $\theta_2^*>0$ is defined in \eqref{eq_theta_2_lim}), there exist $\delta \in \left] 0 , \nicefrac{R_0}{2 \sqrt{d}} \right[$, $K>0$ and $\rho_0>0$ such that for all contour $\gamma$ and $\beta>0$
		\begin{align*}
			I_{\gamma} &\leq g_0^{|\overline{\gamma}_0|} g_1^{|\overline{\gamma}_1|} e^{-\beta \rho_0 |\overline{\gamma}|}
		\end{align*}
		and 
		\begin{align*}
			\left| \frac{\partial I_\gamma}{\partial z} \right| \leq \left(1 + \frac{K}{1-e^{-z\delta^d}}\right) |\overline{\gamma}|\delta^d g_0^{|\overline{\gamma}_0|} g_1^{|\overline{\gamma}_1|} e^{- \beta \rho_0  |\overline{\gamma}|}.
		\end{align*}
	\end{proposition}
	
	Before going to the proof of Proposition \ref{prop_peierls}, we need some intermediate geometrical results given in the four next lemmas. We start with the following observation. For any configuration $\omega \in \Omega$, if there exist $(i,j) \in (\Z^d)^2$ where $d_\infty(i,j)=1$ such that $\sigma(\omega, i)=1$ and $\sigma(\omega,j)=0$, we call this pair a domino. We assume that $ R_0  \geq 2\delta \sqrt{d} >0$. Therefore for any configuration $\omega$ and any domino $(i,j)$ we have $H_i(\omega)= H_j(\omega) = \delta^d$ because of the presence of a point inside the tile $T_i$ ensuring that the tiles $T_i$ and $T_j$ are covered. In particular the energy of this empty tile is positive and we call this property the energy from vacuum. We need to show that the proportion of such pair of tiles in the contour is at least proportional to the volume of the contour. 
	
	\begin{lemma}\label{lem-domino-domino}
		There exists $r_0 >0$ such that for any contour $\gamma$, the set of dominoes 
		\begin{equation*}
			D(\gamma) := \{ (i,j) \in \overline{\gamma}^2, d_\infty(i,j)=1, \either_i = 1, \either_j=0 \}
		\end{equation*}
		satisfies 
		\begin{equation*}
			| D(\gamma) | \geq r_0 |\overline{\gamma}|.
		\end{equation*}
	\end{lemma}
	
	\begin{proof}
		We start by choosing randomly in a contour $\gamma$ a site $k$ such that $\either_k= 1$. Since it is in a contour, it is non-correct, meaning that there is a site $j \in \gamma$, where $\either_j = 0$ and $d_2(k,j) \leq L$. We choose such $j$ such as it is the closest to $k$. Forcibly we have a site $i$ adjacent to $j$ such that $\either_i= 1$ ( at least in the direction of $k$). And we assign $S_1 =\{ k \}$ and $D_1 = \{ (i,j) \}$. We repeat the process to build $S_{n+1}$ and $D_{n+1}$ by choosing the points inside $\gamma \backslash \bigcup_{k \in S_n} B(k, 4L)$. There is $p \in \N$, the number of step until the process stops because there is a finite number of sites with the spin equal to 1 in a contour. We define $S(\gamma) = S_p$ and $D(\gamma) = D_p$. We know that at this point that $$\overline{\gamma}_1 := \{ i \in \overline{\gamma}, \either_i=1\} \subset \bigcup_{k \in S(\gamma)} B(k,4L) \cap \Z^d $$ and that by non-correctness of sites with spin 0 in the contour we have $$\overline{\gamma}_0 := \{ i \in \overline{\gamma}, \either_i=0\} \subset \bigcup_{k \in S(\gamma)} B(k, 5L) \cap \Z^d.$$ In summary we have
		\begin{equation*}
			\overline{\gamma} \subset \bigcup_{k \in S} B(k, 5L) \cap \Z^d.
		\end{equation*}
		Therefore the cardinals verify the following inequalities
		\begin{equation*}
			|\overline{\gamma}| \leq |S| |B(0,5L) \cap \Z^d|. 
		\end{equation*}
		By construction we have $|S(\gamma)| = |D(\gamma)|$. Hence the inequality we are interested in 
		\begin{equation*}
			|D(\gamma)| \geq r_0 |\overline{\gamma}|   \quad \text{where } \: r_0 = \frac{1}{|B(0,5L) \cap \Z^d|}.
		\end{equation*}
	\end{proof}
	%Let's note that for any $(i,j),  (i', j') \in D$ we have $B(i,2L) \cap B(i', 2L) = \emptyset$. Therefore 
	
	%This last part will be useful when we need to treat each point of D with independence as each areas would be to far apart to have any interaction. (CHECK si on a réellement indépendance ca peut impacter la borne sur L dans partie 3 des contours.) 
	
	For any contour $\gamma$ and any configuration $\omega$ that achieves this contour, using the dominoes we are able to find a non-negligible amount of empty tiles that is covered by the halo of the configuration. Another way to find such tiles in the contour is by counting the tiles covered by the halo which are close to the boundary of the halo. Indeed those tiles are guaranteed to be empty otherwise the boundary would be further away.
	
	\begin{lemma}\label{lem-surface-domino}
		Let $ R_0  \geq 2\delta \sqrt{d} >0$, and let us define $\theta_1^\delta$ such as
		\begin{equation*}
			\theta_1^\delta :=  \inf_{\substack{\omega \in \Omega_f \\ \gamma : \mathcal{S}_{\overline{\gamma}}(L(\omega)) >0 }} \left\{ \frac{ V_{\omega, \gamma, \delta}}{\mathcal{S}_{\overline{\gamma}}(L(\omega))} \right\} 
		\end{equation*}
		and 
		\begin{equation*}
			V_{\omega, \gamma, \delta} = \max \left\{\mathcal{V}(T_I), I \subset \overline{\gamma}, \forall i \in I, T_i \subset \partial L(\omega) \oplus B(0,R_0) \cap L(\omega) \right\}.
		\end{equation*}
		We have $\theta_1^\delta >0$ and  $\theta_1^\delta \underset{\delta \rightarrow 0}{\rightarrow} \theta_1^*$ where $\theta_1^* = R_0 \frac{\mathcal{V}(B(0,1))}{\mathcal{S}(B(0,1))}$. 
	\end{lemma}
	
	\begin{proof}
		For any finite configuration $\omega \in \Omega_f$ and $\gamma$ a contour that is created by this configuration such that $\mathcal{S}_{\overline{\gamma}}(L(\omega)) >0$, we have the following inequalities
		\begin{equation*}
			V_{\omega, \gamma, \delta}^{+} \geq V_{\omega, \gamma, \delta} \geq V_{\omega, \gamma, \delta}^{-}
		\end{equation*}
		where 
		\begin{align*}
			&V_{\omega, \gamma, \delta}^{+} = \mathcal{V}(\partial L(\omega) \oplus B(0,R_0) \cap L(\omega) \cap \widehat{\gamma}) \\
			&V_{\omega, \gamma, \delta}^{-} = \mathcal{V}((\partial L(\omega) \oplus B(0,R_0- \delta)) \backslash (\partial L(\omega) \oplus B(0,\delta)) \cap L(\omega) \cap \widehat{\gamma}).
		\end{align*}
		The boundary of the halo $L(\omega)$ inside $\overline \gamma$, appearing in the computation of $\mathcal{S}_{\overline{\gamma}}(L(\omega))$	, is the union of spherical caps built via some marked points $(x_1,r_1), \dots, (x_m, r_m) \in \omega$. We denote by $\alpha_i\in[0,1]$ the ratio of the surface of the $i$th spherical cap with respect to the total surface of its sphere. Therefore, by a simple geometrical argument 
		\begin{align*}
			\frac{V_{\omega, \gamma, \delta}^{-}}{\mathcal{S}_{\overline{\gamma}}(L(\omega))} &\geq \frac{\mathcal{V}(B(0,1))}{\mathcal{S}(B(0,1))} \left( \frac{\sum_{i=1}^{m} \alpha_i (r_i^d - (r_i-R_0)^d)}{\sum_{i=1}^{m} \alpha_i r_i^{d-1}} - \epsilon_{\omega,\gamma}(\delta) \right) \\
			&  \geq \frac{\mathcal{V}(B(0,1))}{\mathcal{S}(B(0,1))} ( R_0 - \epsilon_{\omega,\gamma}(\delta))
		\end{align*}
		where 
		\begin{align*}
			\epsilon_{\omega,\gamma}(\delta) = \frac{\sum_{i=1}^{m} \alpha_i (r_i^d - (r_i-\delta)^d + (r_i - R_o + \delta)^d - (r_i-R_0)^d)}{\sum_{i=1}^{m} \alpha_i r_i^{d-1}}.
		\end{align*}
		Therefore for all $\omega \in \Omega_f$ and the associated contour $\gamma$ we have 
		\begin{align*}
			\liminf_{\delta \to 0} \frac{V_{\omega, \gamma, \delta}}{\mathcal{S}_{\overline{\gamma}}(L(\omega))} \geq \frac{R_0 \mathcal{V}(B(0,1))}{\mathcal{S}(B(0,1))} \implies \liminf_{\delta \to 0} \theta_1^\delta \geq \theta_1^*.
		\end{align*}
		Note that for $\omega = \{(0,R_0)\}$ we have $ \frac{V_{\omega, \gamma, \delta}^{+}}{\mathcal{S}_{\overline{\gamma}}(L(\omega))} = \frac{R_0 \mathcal{V}(B(0,1))}{\mathcal{S}(B(0,1))}$. So
		\begin{align*}
			\frac{R_0 \mathcal{V}(B(0,1))}{\mathcal{S}(B(0,1))} \geq \inf_{\substack{\omega \in \Omega_f \\ \gamma : \mathcal{S}_{\overline{\gamma}}(L(\omega)) >0 }} \left\{ \frac{ V_{\omega, \gamma, \delta}^{+}}{\mathcal{S}_{\overline{\gamma}}(L(\omega))} \right\} 
		\end{align*}
		which implies that $\theta_1^* \geq \limsup_{\delta \to 0} \theta_1^\delta $.
	\end{proof}
	
	\begin{lemma}\label{lem-chi}
		Let $\delta L \geq 2 R_1$ and $d=2$. For any contour $\gamma$ and any configuration $\omega$ that achieves this contour we have 
		\begin{equation*}
			\chi_{\overline{\gamma}}(\omega) \leq \frac{|\overline{\gamma}| \delta^d}{R_0^{d}\mathcal{V}(B(0,1))}.
		\end{equation*}
	\end{lemma}
	
	\begin{proof}
		By definition of contours and the conditions on $\delta$ and $L$ we have
		\begin{align*}
			\chi_{\overline{\gamma}}(L(\omega)) = \chi_{\overline{\gamma}}(L(\omega_{\widehat{\gamma}})) \leq N_{cc}(L(\omega_{\widehat{\gamma}})).
		\end{align*}
		Now the aim is to find for each connected components $C$ of $L(\omega_{\widehat{\gamma}})$ a single point $(x,R) \in \omega_{\widehat{\gamma}}$ such that the ball $B(x,R) \subset C\cap  \widehat{\gamma}$. Note that a ball  $B(x,R)\subset C$ is not necessary included inside the contour. If there exists such a ball not including in the contour, then there is $A \subset \overline{\gamma}^c$ a connected component such that $B(x,R) \cap \widehat{A} \neq \emptyset$. It gives the information that the site $i \in \Z^d$ such that $x \in T_i$ is included inside $\partial_{ext} A$ and by Lemma \ref{lem_partial_type} for all  $j \in \partial_{ext}A$ we have $\sigma(j, \omega_{\omega_{\widehat{\gamma}}}) = 1$. Therefore all balls that are in the tiles corresponding to $\partial_{ext}A$ belong to the same connected component of the halo. We choose a site $j \in \partial_{ext}A$ such that $d_2(j, A) = \lceil \frac{2R_1}{\delta} \rceil$ and so there exists $(y,R') \in \omega_{T_j}$ such that $B(y,R') \subset \widehat{\gamma}$. Therefore we can replace the original representative of the connected component with one that is more suitable.  
		
		With this procedure we have now built, for each connected components $C$ of $L(\omega_{\widehat{\gamma}})$, a single point $(x,R) \in \omega_{\widehat{\gamma}}$ such that the ball $B(x,R) \subset C\cap  \widehat{\gamma}$. We define $I(\omega_{\widehat{\gamma}})$ as the set of all these points which represent the connected components of $L(\omega_{\widehat{\gamma}})$. By construction for any $(x,R) \neq (y,R') \in I(\omega_{\widehat{\gamma}}), B(x,R) \cap B(y,R') = \emptyset$ and therefore we have
		\begin{equation*}
			N_{cc}(\omega_{\widehat{\gamma}}) = |I(\omega_{\widehat{\gamma}})| \leq \frac{|\overline{\gamma}| \delta^d}{\mathcal{V}(B(0,R_0))}.
		\end{equation*}
		
	\end{proof}
	
	\begin{lemma}\label{lem-ratio-0-1}
		For any $\either \in \{0,1\}$ and any contour $\gamma$ we have $r_1 |\overline{\gamma} | \leq  |\overline{\gamma}_\either| \leq (1-r_1) |\overline{\gamma} |$ where
		\begin{equation*}
			r_1 = \frac{1}{|B(0, 2L)\cap \Z^d|} \quad  \text{and  } \quad \overline{\gamma}_\either =  \{ i \in \overline{\gamma}, \either_i = \either\}.
		\end{equation*} 
	\end{lemma}

	\begin{proof}
		Let us set a contour $\gamma$. Let us define $\phi : \overline{\gamma}_1 \rightarrow \overline{\gamma}_0$ such that for all $i \in  \overline{\gamma}_1$, $\phi(i)=j$ where $j$ is the closest site in $\overline{\gamma}_0$ from $i$. A lexicographical procedure is used if equality.
		\begin{align*}
			\alpha_\gamma^1 := \frac{|\overline{\gamma}_1|}{|\overline{\gamma}|} 
			\iff \frac{\alpha_\gamma^1}{1-\alpha_\gamma^1}  = \frac{|\overline{\gamma}_1|}{|\overline{\gamma}_0|} =  \frac{\sum_{j \in \overline{\gamma}_0 } |\phi^{-1}(j)|  }{|\overline{\gamma}_0|} \leq \max |\phi^{-1}(j)|.
		\end{align*}
		It is easy to see that $\max |\phi^{-1}(j)| \leq |B(0, 2L)\cap \Z^d|-1$ and using the fact that $\alpha_\gamma^1 + \alpha_\gamma^0 =1 $ we have that  	
		\begin{align*}
			\alpha_\gamma^1 \leq 1 - \frac{1}{|B(0, 2L)\cap \Z^d|} \\
			\alpha_\gamma^0 \geq \frac{1}{|B(0, 2L)\cap \Z^d|}.
		\end{align*}
		By symmetry of the roles we can exchange the value of the spin. 
	\end{proof}
	
	We have now all the lemmas for proving the Proposition \ref{prop_peierls}.
	
	\begin{proof}
		 We detail the proof of Proposition \ref{prop_peierls} in dimension $d=2$. In higher dimension, the proof works in the same manner but we assume that $\theta_2 = 0$. Let us set $L = \ceil{\frac{2R_1}{\delta}}$, therefore the constant $r_0$ in Lemma \ref{lem-domino-domino} has the following expression
		\begin{align*}
			r_0(\delta) = \frac{1}{\left|B\left(0,5\ceil{\frac{2R_1}{\delta}}\right) \cap \Z^d \right|}.
		\end{align*}
		Let us consider $\theta_1 >- \theta_1^*$, we define $\theta_2^*(\theta_1)$ and $\theta_2^\delta(\theta_1)$ as 
		\begin{align}\label{eq_theta_2_lim}
			&\theta_2^*(\theta_1) = \begin{cases}
				r_0\left(\frac{R_0}{2 \sqrt{d}}\right) \mathcal{V}(B(0,R_0)) & \text{when } \theta_1 \geq 0 \\
				\sup_{\delta \in \left]0, \frac{R_0}{2 \sqrt{d}}\right[  : \theta_1 > -\theta_1^\delta} \left\{r_0(\delta) \mathcal{V}(B(0,R_0))(1 + \frac{\theta_1}{\theta_1^\delta})\right\} & \text{when } \theta_1 < 0
			\end{cases} \\ 	
			&\theta_2^\delta(\theta_1) = \begin{cases}
				r_0\left(\frac{R_0}{2 \sqrt{d}}\right) \mathcal{V}(B(0,R_0)) & \text{when } \theta_1 \geq 0 \\
				\left\{r_0(\delta) \mathcal{V}(B(0,R_0))(1 + \frac{\theta_1}{\theta_1^\delta})\right\} & \text{when } \theta_1 < 0 \nonumber
			\end{cases}.
		\end{align}
		We know by Lemma \ref{lem-surface-domino} that for sufficiently small $\delta$ we have $\theta_1 \geq -\theta_1^\delta > -\theta_1^*$ and $\theta_2 \leq \theta_2^\delta(\theta_1) < \theta_2^*(\theta_1)$. In the case where $\theta_1 <0$ we need to consider a threshold t such that 
		\begin{align*}
			\frac{\theta_2}{\theta_1^\delta + \theta_1} \frac{\delta^d}{\mathcal{V}(B(0,R_0))} < t < \frac{1}{\theta_1}\left(\frac{\theta_2 \delta^d}{\mathcal{V}(B(0,R_0))} - r_0(\delta)\delta^d\right).
		\end{align*}
		We define the quantity $\rho_0$ as such
		\begin{align*}
			\rho_0 = \begin{cases}
				r_0(\delta) \delta^d - \frac{\theta_2 \delta^d}{\mathcal{V}(B(0,R_0))} & \text{when } \theta_1 \geq 0 \\
				\min \left\{ (\theta_1^\delta + \theta_1)t - \frac{\theta_2 \delta^d}{\mathcal{V}(B(0,R_0))} , \; r_0(\delta) \delta^d + \theta_1 t - \frac{\theta_2 \delta^d}{\mathcal{V}(B(0,R_0))} \right\} & \text{when } \theta_1 < 0
			\end{cases}.
		\end{align*}
		With the conditions on $\delta$ and $t$, it guarantees that $\rho_0>0$. Under our assumptions, for any configuration $\omega \in \Omega_f$ that achieves the contour $\gamma$ we claim that 
		\begin{align}\label{eq_peierls_energy}
			H_{\overline{\gamma}} (\omega) \geq |\overline{\gamma}_1| \delta^d + \rho_0 |\overline{\gamma}|.
		\end{align}	
		First let us prove the claim in the case where $\theta_1 <0$. With the condition $R_0 > 2\delta \sqrt{d}$ we know that we have a non negligible amount of empty tiles that are completely covered by the halo. Therefore using Lemmas \ref{lem-domino-domino} and \ref{lem-surface-domino} we have the following two lower bounds on the energy of the contour 	
		\begin{align*}
			H_{\overline{\gamma}} (\omega) & \geq \begin{cases}
				|\overline{\gamma}_1| \delta^d + r_0 |\overline{\gamma}| \delta^d + \theta_1 \mathcal{S}_{\overline{\gamma}}(L(\omega)) - \theta_2 \chi_{\overline{\gamma}}(L(\omega)) \\
				|\overline{\gamma}_1| \delta^d + (\theta_1^\delta  + \theta_1 )\mathcal{S}_{\overline{\gamma}}(L(\omega)) - \theta_2 \chi_{\overline{\gamma}}(L(\omega)) 
			\end{cases} \\
			& \geq \begin{cases}
				|\overline{\gamma}_1| \delta^d + r_0 |\overline{\gamma}| \delta^d + \theta_1 \mathcal{S}_{\overline{\gamma}}(L(\omega)) - \theta_2 \frac{|\overline{\gamma}| \delta^d}{\mathcal{V}(B(0,R_0))} \\
				|\overline{\gamma}_1| \delta^d + (\theta_1^\delta  + \theta_1 )\mathcal{S}_{\overline{\gamma}}(L(\omega)) - \theta_2 \frac{|\overline{\gamma}| \delta^d}{\mathcal{V}(B(0,R_0))}
			\end{cases} \text{ by Lemma \ref{lem-chi}}.
		\end{align*}  
		Depending on the value of the surface inside the contour, one lower bound will be more preferable than the other. Since $\theta_1^\delta + \theta_1>0$ and given the threshold $t$ that verifies our assumption we have   	
		\begin{align*}
			H_{\overline{\gamma}} (\omega) & \geq \begin{cases}
				|\overline{\gamma}_1| \delta^d + \left(r_0 \delta^d + \theta_1 t - \frac{\theta_2 \delta^d}{\mathcal{V}(B(0,R_0))} \right) |\overline{\gamma}| & \text{ if }  \mathcal{S}_{\overline{\gamma}}(L(\omega)) \leq t |\overline{\gamma}|\\
				|\overline{\gamma}_1| \delta^d + (\theta_1^\delta  + \theta_1 ) t |\overline{\gamma}| - \theta_2 \frac{|\overline{\gamma}| \delta^d}{\mathcal{V}(B(0,R_0))} & \text{ if }  \mathcal{S}_{\overline{\gamma}}(L(\omega)) > t |\overline{\gamma}|
			\end{cases}.
		\end{align*}
		In either cases, we have the desired lower boundary on the energy of a contour. Let us turn to the second case where $\theta_1 \geq 0$. It is even easier since we can simply drop the contribution of the surface in the energy and therefore we have
		\begin{align*}
			H_{\overline{\gamma}} (\omega) & \geq 
			|\overline{\gamma}_1| \delta^d + \left( r_0 \delta^d - \frac{\theta_2 \delta^d}{\mathcal{V}(B(0,R_0))}\right) |\overline{\gamma}| = |\overline{\gamma}_1| \delta^d + \rho_0 |\overline{\gamma}|. 
		\end{align*}
		This finishes the proof of the claim. We can now use inequality \eqref{eq_peierls_energy} to get the following upper bound
		\begin{align*}
			I_\gamma &\leq e^{-\beta(\delta^d |\overline{\gamma}_1| + \rho_0|\overline{\gamma}|) } \int \mathbbm{1}_{\left(\forall i \in \overline{\gamma}, \sigma(\omega,i)=\sigma_i \right)} \Pi_{\widehat{\gamma}}^z(d\omega) \\
			& \leq e^{-\beta(\delta^d |\overline{\gamma}_1| + \rho_0|\overline{\gamma}|) } (1-e^{z\delta^d})^{|\overline{\gamma}_1|} e^{z\delta^d |\overline{\gamma}_0|} \\
			& \leq g_0^{|\overline{\gamma}_0|} g_1^{|\overline{\gamma}_1|} e^{-\beta \rho_0 |\overline{\gamma}|}.
		\end{align*}
		By a direct computation we have 
		\begin{align*}
			\frac{\partial I_\gamma}{\partial z} = -|\overline{\gamma}|\delta^d I_\gamma + \frac{1}{z} \int N_{\widehat{\gamma}}(\omega) e^{-\beta H_{\overline{\gamma}}(\omega)} \mathbbm{1}_{\{ \forall i \in \overline{\gamma}, \sigma(\omega, i) = \sigma_i \}} \Pi_{\widehat{\gamma}}^z (d\omega).
		\end{align*}
		Using the bound on $I_\gamma$ we have the desired control on the first term. For the second term we need to control the mean number of points in the contour. By using the lower bound on the energy in a contour \eqref{eq_peierls_energy} and Lemma \ref{lem-ratio-0-1} we have
		\begin{align*}
			\int N_{\widehat{\gamma}}(\omega) e^{-\beta H_{\overline{\gamma}}(\omega)}\mathbbm{1}_{\left(\forall i \in \overline{\gamma}, \sigma(\omega,i)=\sigma_i \right)} \Pi_{\widehat{\gamma}}^z(d\omega) & \leq e^{-\beta ( \rho_0|\overline{\gamma}| + |\overline{\gamma}_1|\delta^d ) }  \int N_{\widehat{\gamma}}(\omega)\mathbbm{1}_{\left(\forall i \in \overline{\gamma}, \sigma(\omega,i)=\sigma_i \right)} \Pi_{\widehat{\gamma}}^z(d\omega) \\
			& \leq e^{-\beta ( \rho_0 |\overline{\gamma}| + |\overline{\gamma}_1|\delta^d)  } \sum_{i \in \overline{\gamma}_1} \int N_{T_i}(\omega) \mathbbm{1}_{\left(\forall i \in \overline{\gamma}, \sigma(\omega,i)=\sigma_i \right)} \Pi_{\widehat{\gamma}}^z(d\omega) \\
			& \leq e^{-\beta ( \rho_0|\overline{\gamma}| + |\overline{\gamma}_1|\delta^d)  } g_0^{|\overline{\gamma}_0|} |\overline{\gamma}_1| z \delta^d  (1-e^{-z\delta^d})^{|\overline{\gamma}_1|-1} \\
			& \leq \frac{(1-r_1)z \delta^d}{1- e^{-z\delta^d}}|\overline{\gamma}| g_0^{|\overline{\gamma}_0|} g_1^{|\overline{\gamma}_1|} e^{- \beta \rho_0  |\overline{\gamma}|}.
		\end{align*}
		By combining those inequalities we get
		\begin{align*}
			\left| \frac{\partial I_\gamma}{\partial z} \right| \leq \left(1 + \frac{1-r_1}{1-e^{-z\delta^d}}\right) |\overline{\gamma}|\delta^d g_0^{|\overline{\gamma}_0|} g_1^{|\overline{\gamma}_1|} e^{- \beta \rho_0  |\overline{\gamma}|}.
		\end{align*}
	\end{proof}
	
	\subsection{Truncated weights and pressure} \label{subsection4.5}
	
	In order to have $\tau$-stable contours, it is needed that the ratio of partition functions of the interior has to be small. This ratio is more likely to be large when the volume of the interior is large. Therefore depending on the parameters $z,\beta$ the weights of some contours might be unstable because the volume of the interior is too large. In order to deal with those instability we will truncate the weights. We follow mainly the ideas and the presentation given in Chapter 7 of \cite{Velenik}. We introduce an arbitrary cut-off function $\kappa: \R \rightarrow [0,1]$ that satisfy the following properties : $\kappa(s) = 1$ if $s \leq \frac{\rho_0}{8} $, $\kappa(s) = 0$ if $s \geq \frac{\rho_0}{4} $ and $\kappa$ is $\mathcal{C}^1$. Therefore such cut-off function $\kappa$ satisfies $\| \kappa' \| = \sup_{\R} |\kappa'(s)| < + \infty $.
	
	We construct step by step the truncated weights and pressure according to the volume of the interior. Contours with no interior does not have a ratio of partition functions introduced in the expression of its weight, thus those contours are stable. We can define the truncated quantities associated to $n=0$ starting by the truncated pressure
	\begin{align*}
		\widehat{\psi}_0^\either := \frac{\ln(g_\either)}{\beta \delta^d} .
	\end{align*}
	We define the truncated weights for contour $\gamma$ of class $0$ as
	\begin{align*}
		\widehat{w}_{\gamma}^\either = w_{\gamma}^\either = g_\either^{-|\overline{\gamma}|} I_{\gamma}.
	\end{align*}
	Now we suppose that the truncated weights are well defined for contours $\gamma$ of class $k \leq n$. We define $\widehat{\Phi}_n^\either$ with the following polymer development
	\begin{align*}
		\widehat{\Phi}_n^\either (\Lambda) := \sum_{\Gamma \in \mathcal{C}_n^\either(\Lambda) } \prod_{\gamma \in \Gamma} \widehat{w}_{\gamma}^\either.
	\end{align*}
	We can define the truncated partition function as
	\begin{align*}
		\widehat{Z}_n^\either (\Lambda) := g_\either^{|\Lambda|} \widehat{\Phi}_n^\either (\Lambda).
	\end{align*}
	With $\Lambda_k = [-\nicefrac{k}{2}, \nicefrac{k}{2}]^d \cap \Z^d$, the truncated pressure at rank n is given by 
	\begin{align*}
		\widehat{\psi}_n^\either & := \lim_{k \rightarrow +\infty} \frac{1}{\beta \delta^d |\Lambda_k|} \ln(\widehat{Z}_n^\either (\Lambda_k)) \\ 
		& = \widehat{\psi}_0^\either + \lim_{k \rightarrow +\infty} \frac{1}{\beta \delta^d |\Lambda_k|} \ln(\widehat{\Phi}_n^\either (\Lambda_k)).
	\end{align*}
	This limit exists as it is shown in Section 7.4.3 of \cite{Velenik}. 
	We can also note that $\widehat{\psi}_n^\either \in [\widehat{\psi}_0^\either, \psi]$ and that the sequence of truncated pressure $(\widehat{\psi}_n^\either)_{n \in \N}$ is increasing.

	\begin{definition}
		Given the truncated weight of contours of class $k \leq n$ (and so the truncated pressure $\widehat{\psi}_n^\either$), the truncated weight of a contour $\gamma$ of class $n+1$ is defined by 
		\begin{align*}
			\widehat{w}_{\gamma}^\either = g_\either^{-|\overline{\gamma}|} I_{\gamma} \, \kappa \left( (\widehat{\psi}_n^{\either^*} - \widehat{\psi}_n^\either) \delta^d |\Int_{\either^*} \gamma|^{\frac{1}{d}}  \right) \frac{Z_{\Int_{\either^*} \gamma}^{\either^*}}{Z_{\Int_{\either^*} \gamma}^{\either}}.
		\end{align*}
	\end{definition}
	
	Intuitively, the reason we do this cut-off is to remove contours that are unstable. Since the cut-off function verify $0 \leq \kappa \leq 1 $ we have $\widehat{w}_{\gamma}^\either  \leq w_{\gamma}^\either$. Other quantities that are of interest for what will come later are, $\widehat{\psi}_n := \max( \widehat{\psi}_n^0, \widehat{\psi}_n^1)$ and $a_n^\either := \widehat{\psi}_n - \widehat{\psi}_n^\either$. By definition, we have $a_n^\either \geq 0$ and for all contour $\gamma$ of class $n+1$ we have the following implication
	\begin{align*}
		a_n^\either \delta^d (n+1)^{\frac{1}{d}} \leq \frac{\rho_0}{8} \implies \widehat{w}_{\gamma}^\either  = w_{\gamma}^\either.
	\end{align*}	
	Before we go further we will proceed into a re-parametrisation of the model for fixed $\theta_1$ and $\theta_2$ that verify the assumption of Proposition \ref{prop_peierls}. We change the parameters from $(\beta, z)$ to $(\beta, s)$ where $s = z/\beta$.
	
	We set $a = \min \{ \frac{2}{1-r_1}, e^{-c \beta}\}$ where $r_1$ is given in Lemma \ref{lem-ratio-0-1} and $c>0$. For all $\beta > 0$ we set 
	$$ s_\beta = \frac{\ln(1+e^{\beta \delta^d})}{\beta \delta^d}$$ 
	such that we have $g_0 =g_1$. Furthermore, we define the following open interval $$U_\beta = \left( \frac{\ln(1+e^{\beta \delta^d-a})}{\beta \delta^d}, \frac{\ln(1+e^{\beta \delta^d + a})}{\beta \delta^d} \right)$$ such that for all $s \in U_\beta$ we have $ e^{-a} \leq \frac{g_0}{g_1} \leq e^{a}$. So, according to Proposition \ref{prop_peierls}, for any $\beta>0$ and $s \in U_\beta$ we have 
	\begin{equation*}
		w_\gamma^\either \leq e^{-(\beta \rho_0 -2)|\overline{\gamma}|} \frac{Z_{\Int_{\either^*} \gamma}^{\either^*}}{Z_{\Int_{\either^*} \gamma}^{\either}}.
	\end{equation*}
	We can already see that for contours of class $0$ the weights are $\tau$-stable as long as $\beta \rho_0 > 2$. In the following proposition we will show that truncated weights are $\tau$-stable and that we have a good bound on the derivative of the truncated weights that are useful in the study of the regularity of the pressure.
	
	\begin{proposition} \label{prop_tau_stability_truncated_weights}
		Let $\tau := \frac{1}{2}\beta  \rho_0 - 8$. There exists $D\geq 1$ and  $0<\beta_0< \infty$ such that for all $\beta > \beta_0$ there exist $C_1>0$ and $C_2>0$ where the following statements hold for any $\either$ and $n\ge 0$.
		
		\begin{enumerate}
			\item (Bounds on the truncated weights) For all $k \leq n$, the truncated weights of each contours $\gamma$ of class $k$ is $\tau$-stable uniformly on $U_\beta$ :
			\begin{equation}
				\widehat{w}_{\gamma}^\either \leq e^{-\tau |\overline{\gamma}|} \label{truncated_weights_tau_stability}
			\end{equation}
			and 
			\begin{equation} \label{implication_bound_a_n_to_stability}
				a_n^\either \delta^d |\Int \gamma|^{\frac{1}{d}} \leq \frac{\rho_0 }{16} \implies \widehat{w}_{\gamma}^\either = w_{\gamma}^\either. 
			\end{equation}
			Moreover, $s \mapsto \widehat{w}_{\gamma}^\either $ is $\mathcal{C}^1$ and uniformly on $U_\beta$ it verifies
			\begin{equation}
				\left|\frac{\partial \widehat{w}_{\gamma}^\either}{\partial s} \right| \leq D |\overline{\gamma}|^{\frac{d}{d-1}} e^{-\tau |\overline{\gamma}|}. \label{bound_truncated_weight_derivative}
			\end{equation}
			\item (Bounds on the partition functions) Let us assume that $\Lambda \subset \Z^d$ and $|\Lambda| \leq n+1$. Then uniformly on $U_\beta$ we have
			\begin{align}
				Z_\Lambda^\either& \leq e^{\beta \delta^d \widehat{\psi}_{n} |\Lambda| + 2 |\partial_{ext} \Lambda|}, \label{bound_partition_function}\\
				\left| \frac{\partial Z_\Lambda^\either}{\partial s} \right| & \leq \left( C_1 \beta \delta^d |\Lambda| + C_2 |\partial_{ext} \Lambda| \right) e^{\beta \delta^d \widehat{\psi}_n |\Lambda| + 2 |\partial_{ext} \Lambda|}. \label{bound_derivative_partition_function}
			\end{align}
		\end{enumerate}
	\end{proposition}
	
	This proposition is the key tool for applying the PSZ theory and is very similar to Proposition 7.34 in \cite{Velenik} with some adaptations to the present setting. Due to these small adaptations and also to be self-contained, the proof of the proposition is given in the Annex B.

	\subsection{Proof of Proposition \ref{prop_first_order_transition}} \label{subsection4.6}
	
	According to Proposition \ref{prop_tau_stability_truncated_weights} we know that the weights are $\tau$-stable with $\tau$ as large as we want provided that $\beta$ is large enough. So for $\beta$ large enough, for any $\Lambda \subset \Z^d$ we can define 
	\begin{equation*}
		\widehat{\Phi}^\either(\Lambda) := \sum_{\Gamma \in \mathcal{C}^\either(\Lambda)} \prod_{\gamma \in \Gamma} \widehat{w}^\either_\gamma
	\end{equation*}
	and based on standard cluster expansion results, recalled in Annex A, it is also possible to have a polymer development of $\log(\widehat{\Phi}^\either(\Lambda))$. In particular by Theorem \ref{theorem_cluster_expansion} we have
	\begin{align*}
		\widehat{\Phi}^\either(\Lambda) = e^{\beta \delta^d f^\either |\Lambda| + \Delta^\either(\Lambda)}	
	\end{align*}
	where $f^\either$ and $\Delta^\either_\Lambda$ are $\mathcal{C}^1$ in $U_\beta$. Furthermore, according to Theorem \ref{theorem_cluster_expansion} we also have
	\begin{align*}
		&|f^\either| \leq  \eta(\tau, l_0),&  &|\Delta^\either_\Lambda| \leq \eta(\tau, l_0) |\partial_{ext} \Lambda | \\
		&\left| \frac{\partial f^\either}{\partial s} \right| \leq D \eta(\tau, l_0),&  &\left| \frac{\partial \Delta^\either_\Lambda}{\partial s} \right| \leq D \eta(\tau, l_0) |\partial_{ext} \Lambda|,
	\end{align*}
	where $D>0$ comes from Proposition \ref{prop_tau_stability_truncated_weights} and the constant $\eta(\tau, l_0)$, used intensively in Annex A, is equal to $2e^{- \nicefrac{\tau l_0}{3}}$, where $l_0$ is the smallest size of a non-empty contour (it is a positive integer). Therefore the truncated partition function with boundary $\either$ is given by
	\begin{align}\label{eq_Z_cluster}
		\widehat{Z}_\Lambda^\either := g_\either^{|\Lambda|} \widehat{\Phi}^\either (\Lambda) = e^{\beta \delta^d ( \widehat{\psi}_0^\either + f^\either)|\Lambda| + \Delta^\either(\Lambda) }
	\end{align}
	and the truncated pressure associated to the boundary $\either$ is 
	\begin{align*}
		\widehat{\psi}^\either := \widehat{\psi}_0^\either + f^\either.
	\end{align*}
	We can also obtain $\widehat{\psi}^\either$ as the limit of the sequence of truncated pressure $(\widehat{\psi}_n^\either)_{n \in \N}$ when $n$ is going to infinity.  So we define the truncated pressure as $\widehat{\psi} : = \max\{ \widehat{\psi}^{(0)}, \widehat{\psi}^{(1)}\}$ and we have that $a^\either : = \widehat{\psi} - \widehat{\psi}^\either = \lim_{n \to \infty} a_n^\either$. 
	
	Similarly to the statement \eqref{implication_bound_a_n_to_stability} in Proposition \ref{prop_tau_stability_truncated_weights}, for any contour $\gamma$ we have the following implication 
	\begin{equation*}
		a^\either \delta^d |\Int \gamma |^{\frac{1}{d}} \leq \frac{\rho_0}{16} \implies \widehat{w}_\gamma^\either = w_\gamma^\either.
	\end{equation*}
	Therefore when $a^\either =0$ all the truncated weights are equal to the actual weight of the model and thus for all $\Lambda \subset \Z^d$ we have $\widehat{Z}_\Lambda^\either = Z_\Lambda^\either$ and with Lemma \ref{lem_pression_free_boundary}  we have
	\begin{equation*}
		\psi = \widehat{\psi} = \max \{ \widehat{\psi}^{(1)}, \widehat{\psi}^{(0)}\}.
	\end{equation*}
	Therefore we can extract properties of the pressure from the study of the truncated pressures. For $s \in U_\beta$ let us define the gap $G(s)$ between the two truncated pressures
	\begin{equation*}
		G(s):= \widehat{\psi}^{(1)} - \widehat{\psi}^{(0)} = (s-1) + \frac{\ln(1-e^{-s\beta \delta^d})}{\beta \delta^d} + f^{(1)} - f^{(0)}.
	\end{equation*}
	 Previously, we have set $a =\min \{\frac{2}{1-r_1}, e^{-c\beta} \}$ with the condition that $c>0$. We are going to be more precise and take $c < \frac{1}{6} \rho_0 l_0$. Hence for sufficiently large $\beta$ we have 
	\begin{align*}
		G(s_\beta^{-}) = -\frac{a}{\delta^d} + f^{(1)} - f^{(0)} \leq -\frac{a}{\delta^d} + 2 \eta(\tau, l_0) < 0 \\
		G(s_\beta^{+}) = \frac{a}{\delta^d} + f^{(1)} - f^{(0)} \geq \frac{a}{\delta^d} - 2 \eta(\tau, l_0) > 0
	\end{align*}
	where $s_\beta^{-} = \frac{\ln(1+e^{\beta \delta^d - a})}{\beta \delta^d}$ and $s_\beta^{+} = \frac{\ln(1+e^{\beta \delta^d + a})}{\beta \delta^d}$. Furthermore, for $s \in U_\beta$  and for sufficiently large $\beta$ we have 
	\begin{align*}
		\frac{\partial G}{\partial s}(s)  = \frac{1}{e^{s\beta \delta^d}-1} + \frac{\partial f^{(1)}}{\partial s} - \frac{\partial f^{(0)}}{\partial s} > 1 - 2D\eta(\tau, l_0) > 0.
	\end{align*}
	This ensures the existence of an unique $s_\beta^c \in U_\beta$ such that 
	\begin{align*}
		\widehat{\psi} = \begin{cases}
			\widehat{\psi}^{(0)} & \text{ when } s \in [s_\beta^{-}, s_\beta^c] \\
			\widehat{\psi}^{(1)} & \text{ when } s \in [s_\beta^c, s_\beta^{+}]
		\end{cases}
	\end{align*}
	and also for all $s \in U_\beta$ (and thus also for $s_\beta^c$) we have
	\begin{equation}\label{ineq_derivee_pression}
		\frac{\partial \widehat{\psi}^{(1)}}{\partial s}(s) > \frac{\partial \widehat{\psi}^{(0)}}{\partial s}(s). 
	\end{equation}
	Furthermore, we can observe that 
	\begin{align*}
		z_\beta^- - \beta := \beta s_\beta^- - \beta = -\frac{a}{ \delta^d }  + \frac{1}{\delta^d}\ln(1 + e^{-\beta \delta^d + a }) = -\frac{a}{\delta^d } + o(a(\beta)) \\ 
		z_\beta^+ - \beta := \beta s_\beta^+ - \beta = \frac{a}{ \delta^d }  + \frac{1}{\delta^d}\ln(1 + e^{-\beta \delta^d - a }) = \frac{a}{\delta^d } + o(a(\beta)) 
	\end{align*}
	Therefore we have that $|z_\beta^c
	- \beta | = O(e^{-c\beta})$  when $\beta$ tends to infinity where $z_\beta^c := \beta s_\beta^c$ and $0<c<\nicefrac{\rho_0 l_0}{6}$. This last part proves the statement in Theorem \ref{thm_quermass_transition} about the exponential decay of the difference $|\beta - z_\beta^c|$.  
	\newline
	
	For all $n\geq1$, we define the empirical field $\overline{P}_{\Lambda_n}^\either$ as the probability measure on $\Omega$ such that for any positive test function $f$
	\begin{equation*}
		\int f(\omega) \overline{P}^\either_{\Lambda_n}(d\omega) = \frac{1}{|\widehat{\Lambda}_n| } \int_{\widehat{\Lambda}_n} \int f(\tau_u(\omega)) \widehat{P}^\either_{\Lambda_n}(d\omega)du \quad \text{where } \widehat{P}_{\Lambda_n}^\either = \bigotimes_{i \in \Z^d} P_{\tau_{2ni}(\Lambda_n)}^\either . 
	\end{equation*}
	The empirical field can be seen as a stationarization of $ P^\either_{\Lambda_n}$ and therefore any accumulation point of the sequence $ (\overline{P}_{\Lambda_n}^\either)_{n \in \N}$ is necessarily stationary. Furthermore, we know that the Quermass interaction is stable and has a finite range interaction. Therefore there exists a sub-sequence of the empirical field $(\overline{P}_{\Lambda_n}^\either)_{n \in \N}$ that converges to $P^\either$  for the local convergence topology defined as the smallest topology on the space of probability measures on $\Omega$ such that for all tame functions $f$ (i.e. $f : \Omega \rightarrow \R$ such that $\exists \Delta \subset \R^d$ bounded, $\exists A \geq 0,  \forall \omega \in \Omega, f(\omega)= f(\omega_\Delta)$ and $|f| \leq A(1+N_\Delta)$) the map $P \rightarrow \int fdP$ is continuous. Moreover this limit is a Gibbs measure as it satisfies the DLR equations. The proof of these statements is similar to the proof of Theorem 1 in \cite{MiniCours}. For all $\beta \geq \beta_0$ and for $z = z_\beta^c$ we have by a direct computation 
	\begin{equation}\label{eq_derivee_fct_partition_1}
		\frac{\partial}{\partial z} \log \widehat{Z}_{\Lambda_n}^\either = \frac{\partial}{\partial z} \log Z_{\Lambda_n}^\either = - \delta^d |{\Lambda_n}| + \frac{1}{z}E_{P_{\Lambda_n}^\either}(N_{\widehat{\Lambda}_n}).
	\end{equation}
	On the other hand we know using \eqref{eq_Z_cluster} that 
	\begin{equation}\label{eq_derivee_fct_partition_2}
		\frac{\partial}{\partial z} \log \widehat{Z}_{\Lambda_n}^\either  = \beta |\Lambda_n| \delta^d \frac{\partial \widehat{\Psi}^\either}{\partial z} + \frac{\partial \Delta^\either_{\Lambda_n}}{\partial z}.
	\end{equation}
	By combining \eqref{eq_derivee_fct_partition_1} and \eqref{eq_derivee_fct_partition_2} we get the following relationship
	\begin{equation*}
		\frac{E_{P_{\Lambda_n}^\either}(N_{\widehat{\Lambda}_n})}{|\Lambda_n| \delta^d} = z + z\beta\frac{\partial \psi^\either}{\partial z}(z_\beta^c) + \frac{z}{|\Lambda_n| \delta^d} \frac{\partial \Delta^\either_{\Lambda_n}}{\partial z}.
	\end{equation*}
	Furthermore, using Theorem \ref{theorem_cluster_expansion} we know that
	\begin{equation*}
		\left| \frac{\partial \Delta^\either_{\Lambda_n}}{\partial z} \right| \leq D \eta(\tau, l_0) |\partial_{ext} \Lambda_n | \implies \frac{1}{|\Lambda_n|} \frac{\partial \Delta^\either_{\Lambda_n}}{\partial z} \underset{n \to \infty}{\rightarrow} 0.
	\end{equation*}
	The empirical field is stationary and using the local convergence we have
	\begin{equation*}
		\frac{E_{\overline{P}_{\Lambda_n}^\either}(N_{\widehat{\Lambda}_n})}{|\Lambda_n| \delta^d} = E_{\overline{P}_{\Lambda_n}^\either}(N_{[0,1]^d}) \underset{n \to \infty}{\rightarrow} \rho(P^\either).
	\end{equation*}
	On the other hand, we have by construction of the empirical field that 
	\begin{equation*}
		E_{\overline{P}_{\Lambda_n}^\either}(N_{\widehat{\Lambda}_n}) =  	E_{P_{\Lambda_n}^\either}(N_{\widehat{\Lambda}_n}). 
	\end{equation*}
	In summary, when we take limit when $n$ goes to infinity we have that 
	\begin{equation*}
		\rho(P^\either) = z + z \beta \frac{\partial \psi^\either}{\partial z}(z_\beta^c)
	\end{equation*}
	and therefore using \eqref{ineq_derivee_pression}  we have that  $\rho(P^{(1)}) > \rho(P^{(0)})$. 
	
	\section*{Acknowledgement}
	
	This work was supported in part by the Labex CEMPI (ANR-11-LABX-0007-01), the ANR projects PPPP (ANR-16-CE40-0016) and RANDOM (ANR-19-CE24-0014) and by the CNRS GdR 3477 GeoSto.
	
	\stopcontents[sections]
	
	\section{Annex A : Cluster Expansion } \label{Annex1}

	We consider a collection of contours $\mathcal{C}$ and for all $\Lambda \subset \Z^d$, $\mathcal{C}(\Lambda)$ a sub-collection of contours $\gamma$ such that $\overline{\gamma} \subset \Lambda$. For each contour $\gamma$ in such collection there are weights $w_\gamma$ invariant by translation. We set $l_0 = \min \{ |\overline{\gamma}|, \gamma \in  \mathcal{C} \}$ and $\eta (\tau, l_0) = 2 \exp(-\nicefrac{\tau l_0}{3})$. A set of contours $\Gamma = \{ \gamma_1, \dots, \gamma_n\} \in \mathcal{C}$, is said to be geometrically compatible if for all $i,j \in \{1,\dots, n\}, i \neq j$ we have $d_\infty(\gamma_i, \gamma_j)\geq 1$. We define the polymer development associated to those weights as, for all $\Lambda \subset \Z^d$
	\begin{align*}
		\Phi(\Lambda) = \sum_{\substack{\Gamma \in \mathcal{C}(\Lambda) \\ \text{geometrically compatible}}} \prod_{\gamma \in \Gamma} w_\gamma.
	\end{align*}
	A collection $C=\{\gamma_1, \cdots, \gamma_n \}$ is said to be decomposable if the support $\overline{C} = \bigcup_{\gamma \in C} \overline{\gamma}$ is not simply connected. A cluster, denoted by $X$, is a non-decomposable finite multiset of contours such that a same contour can appear multiple times and we define $\overline{X}: = \bigcup_{\gamma \in X} \overline{\gamma}$. The cluster expansion for $\ln \Phi(\Lambda)$, if it converges, is given by 
	\begin{equation*}
		\ln \Phi(\Lambda) = \sum_{\substack{X : \overline{X} \subset \Lambda}} \Psi(X)
	\end{equation*}
	where $\Psi(X) := \alpha(X) \prod_{\gamma \in X} w_\gamma$. We have this combinatorial term $\alpha$ whose expression is given by $$ \alpha(X)= \left\{\prod_{\gamma \in \mathcal{C}(\Lambda)} \frac{1}{n_X(\gamma)!} \right\} \left\{ \sum_{\substack{G \subset G_n \\ \text{connected} }}  \prod_{\{i,j\} \in G} \zeta(\gamma_i, \gamma_j) \right\} $$ where $n_X(\gamma)$ is the number of times a contour $\gamma$ appears in the cluster $X$, $G_n = (V_n, E_n)$ is an undirected complete graph on $V_n = \{1, \cdots, n\}$ and $\zeta$ is defined as$$ \zeta(\gamma, \gamma') = \begin{cases}
		0 & \text{if } \gamma, \gamma' \text{ are geometrically compatible} \\
		-1 & \text{otherwise}
	\end{cases}.$$ 
	Under the assumption that the weights are $\tau$-stable, a sufficient condition for the convergence of the cluster expansion is 
	\begin{equation*}
		\sum_{\substack{\gamma \in \mathcal{C} \\ 0 \in \overline{\gamma}}} e^{-\tau |\overline{\gamma}|} e^{3^d |\overline{\gamma}|} \leq 1.
	\end{equation*}
	In practice we will take a larger value for $\tau$ such that the stronger assumption of Lemma \ref{lem_condition_convergence_derivative_serie} is verified. The following lemmas correspond to Lemma 7.30 and Lemma 7.31 in \cite{Velenik}.
	
	\begin{lemma}\label{lem_condition_convergence_derivative_serie}
		There exists $\tau_0 >0$ such that when $\tau > \tau_0$
		\begin{equation}
			\sum_{\gamma \in \mathcal{C} : 0 \in \overline{\gamma}} |\overline{\gamma}|^{\nicefrac{d}{d-1}} e^{-(\nicefrac{\tau}{2}-1) |\overline{\gamma}|} e^{3^d |\overline{\gamma}|} \leq \eta(\tau, l_0) \leq 1.
		\end{equation}
	\end{lemma}
	
	\begin{lemma}\label{lemma_cluster_expansion}
		Let us assume that the weights are $\tau$-stable for $\tau > \tau_0$. Then for $L \geq l_0$
		\begin{equation}
			\sum_{\substack{X : 0 \in \overline{X} \\ |\overline{X}| \geq L} } |\Psi(X)| \leq e^{-\frac{\tau L}{2}}.
		\end{equation}
	\end{lemma}
	
	Finally we have the following theorem which plays a crucial role in the proofs of Proposition \ref{prop_first_order_transition} and \ref{prop_tau_stability_truncated_weights}.
	
	\begin{theorem}\label{theorem_cluster_expansion}
		Assume that, for all $\gamma \in \mathcal{C}$, the weight $w_{\gamma}$ is $\mathcal{C}^1$ in a parameter $s\in (a,b)$ and that uniformly on $(a,b)$, 
		\begin{equation}
			w_{\gamma} \leq e^{-\tau |\overline{\gamma}|}, \quad \quad \left| \frac{ d w_{\gamma}}{d s} \right| \leq D|\overline{\gamma}|^{\nicefrac{d}{d-1}} e^{-\tau |\overline{\gamma}|},
		\end{equation}
		where $D \geq 1$ is a constant. Then there exists $\tau_1 = \tau_1(D, d) < \infty$ such that the following holds. If $\tau > \tau_1$, the pressure $g$ is given by the following absolutely convergent series, 
		\begin{equation}
			g = \sum_{X : 0 \in \overline{X}} \frac{1}{|\overline{X}|} \Psi(X)
		\end{equation}
		where the sum is over clusters $X$ made of contours $\gamma \in \mathcal{C}$ and $\overline{X} = \bigcup_{\gamma \in X} \overline{\gamma}$. Moreover,
		\begin{equation*}
			|g| \leq \eta(\tau, l_0) \leq 1
		\end{equation*}
		and for all $\Lambda \subset \Z^d$ finite, $g$ provides the volume contribution to $\log \Phi(\Lambda)$, in the sense that 
		\begin{equation}
			\Phi(\Lambda) = \exp(g|\Lambda| + \Delta_\Lambda)
		\end{equation}
		where $\Delta_\Lambda$ is a boundary term :
		\begin{equation}
			|\Delta_\Lambda| \leq \eta(\tau, l_0)|\partial_{ext} \Lambda|.
		\end{equation}
		Finally, $g$ and $\Delta_\Lambda$ are also $\mathcal{C}^1$ in $s \in (a,b)$; its derivative equals 
		\begin{equation}
			\frac{dg}{ds} = \sum_{X : 0 \in \overline{X}} \frac{1}{|\overline{X}|} \frac{d\Psi(X)}{ds}
		\end{equation}
		and 
		\begin{equation}
			\left| \frac{dg}{ds}\right| \leq D \eta(\tau, l_0), \quad \quad \left| \frac{d\Delta_\Lambda}{ds}\right| \leq D \eta(\tau, l_0) |\partial_{ext} \Lambda|.
		\end{equation}
		
	\end{theorem}
	
	This theorem is similar to Theorem 7.29 in \cite{Velenik}, the only difference being is that the following statement is not included $$ \left|\frac{d \Delta_\Lambda}{d s} \right| \leq D \eta(\tau, l_0) |\partial_{ext} \Lambda|.$$ Following the computations in \cite{Velenik} we find 
	\begin{equation*}
		\frac{d \Delta_\Lambda}{d s} = \sum_{i \in \Lambda} \sum_{X : i \in \overline{X} \not\subset \Lambda} \frac{1}{|\overline{X}|} \frac{d\Psi}{d s}(X).
	\end{equation*}
	Whenever $i \in \overline{X} \not\subset \Lambda$ we know that $\overline{X} \cap \partial_{ext} \Lambda \neq \emptyset$ and since the weights are invariant by translation we have 
	\begin{align*}
		\left|\frac{d \Delta_\Lambda}{d s}\right| \leq |\partial_{ext} \Lambda| \max_{j \in \partial_{ext} \Lambda} \sum_{X : j \in \overline{X}} \left|\frac{d\Psi}{d s}(X) \right| = |\partial_{ext} \Lambda| \sum_{X : 0 \in \overline{X}} \left|\frac{d\Psi}{d s}(X) \right|.
	\end{align*}
	We can show that for any cluster $X$ we have 
	\begin{equation}
		\left|\frac{d \Psi}{d s} (X)\right| \leq |\overline{\Psi}(X)|
	\end{equation}
	where 
	\begin{equation*}
		\overline{\Psi}(X) = \alpha(X) \prod_{\gamma \in X} \overline{w}_\gamma \quad \text{and} \quad \overline{w}_\gamma = D|\overline{\gamma}|^{\nicefrac{d}{d-1}} e^{-(\tau - 1 )|\overline{\gamma}|}.
	\end{equation*}
	With classical results on cluster expansion we can show that 
	\begin{equation*}
		\sum_{X : 0 \in \overline{X}} |	\overline{\Psi}(X)| \leq  \sum_{\gamma \in \mathcal{C} : 0 \in \overline{\gamma}} \overline{w}_\gamma e^{3^d |\overline{\gamma}|}.
	\end{equation*}
	Therefore with Lemma \ref{lem_condition_convergence_derivative_serie} we have the desired control on the derivative of the boundary term.

	\section{Annex B : Proof of Proposition \ref{prop_tau_stability_truncated_weights}} \label{Annex2}
	
	Before starting the proof by induction, we need to fix the some constants and in particular the quantity $\beta$ which has to be sufficiently large. Recall that $\eta(\tau, l_0) := 2e^{-\frac{\tau l_0}{3}}$, where $l_0$ is the minimum size of non empty contour. We set $D := (2+2K)\delta^d \beta  +4\beta + 4 \delta^d \| \kappa' \| + C_1\beta \delta^d + C_2 $ where $ C_1(\beta) := e + \sup_{s \in U_\beta} |\nicefrac{\widehat{\psi}_0^\either}{s}|$ and $ C_2(\beta) := \sup_{s \in U_\beta} \nicefrac{1}{s}$ and we choose $\beta \geq 1$ sufficiently large such that
	\begin{align}
		&\tau > \tau_0(d)  \text{ where $\tau_0$ is defined as in Lemma \ref{lem_condition_convergence_derivative_serie}.} \label{condition_beta_1}\\
		&D \eta(\tau, l_0) \leq 1 \label{condition_beta_2}\\
		&\forall k \in \N, \; 2 \beta^{-1} k^{\nicefrac{1}{d}} \exp{(-\frac{\tau k^{\nicefrac{d-1}{d}}}{2})} \leq \frac{\rho_0 }{16} \label{condition_beta_3}\\
		&\forall x>0, \; \beta^{-1} \delta^{-d} \exp{(-\max\{(\nicefrac{\rho_0 }{16 \delta^d x})^{d-1}, l_0 \} \frac{\tau}{2})} \leq \frac{x}{2}. \label{condition_beta_5}
	\end{align}
	Let us prove the proposition for $n=0$. Let $\overline{\gamma}$ be a contour of class 0. For $s \in U_\beta$ and using the upper bound on $I_{\gamma}$ in Proposition \ref{prop_peierls} we have 
	\begin{align*}
		\widehat{w}_{\gamma}^\either \leq e^{-\tau |\overline{\gamma}|}.
	\end{align*}
	For the derivative we have 
	\begin{equation*}
		\frac{\partial \widehat{w}_\gamma^\either}{\partial z} = \left( -\frac{|\overline{\gamma}|}{g_\either} \frac{\partial g_\either}{\partial z} I_{\gamma} + \frac{\partial I_{\gamma}}{\partial z}\right) g_\either^{-|\overline{\gamma}|}. 
	\end{equation*}
	Therefore using the upper bounds in Proposition \ref{prop_peierls} 
	\begin{equation}\label{ineq_derivative_weights_class0}
		\left| \frac{\partial \widehat{w}_\gamma^\either}{\partial s}  \right| \leq  (2 + 2K )\beta \delta^d |\overline{\gamma}| e^{-(\beta \rho_0 - 2) |\overline{\gamma}|} < D |\overline{\gamma}|^{\nicefrac{d}{d-1}} e^{-\tau |\overline{\gamma}|}.
	\end{equation}
	For the bound on the partition function with $|\Lambda| = 1$ we have
	\begin{equation*}
		Z_\Lambda^\either = g_\either \leq e^{\beta \delta^d \widehat{\psi}_0}.
	\end{equation*}
	Finally concerning the derivative of the partition function with respect to $s$ for sufficiently large $\beta$ we obtain directly by computing that 
	\begin{equation*}
		\left|\frac{\partial Z_\Lambda^\either}{\partial s}\right| \leq \beta \delta^d e^{\beta \delta^d \widehat{\psi}_0} \leq C_1 \beta \delta^d e^{\beta \delta^d \widehat{\psi}_0}.
	\end{equation*}	
	Now we start the induction. Let us assume that the statements have been proven up to $n$. We have to prove that they hold also for $n+1$. 
	
	Since all the contours appearing in $\widehat{\Phi}_n^{\either}$ are at most of class $n$, by the induction hypothesis, all these weights are $\tau$-stable and satisfied the  bound \eqref{bound_truncated_weight_derivative} on the derivatives. Therefore using Theorem \ref{theorem_cluster_expansion} we get 
	\begin{align} \label{bound_terme_perturbatif_pression}
		\left| \frac{\partial f_n^\either}{\partial s} \right| \leq \frac{D \eta(\tau, l_0)}{\beta \delta^d} \leq 1 
	\end{align}
	where $f_n^\either$ appears in $\widehat{\psi}_n^\either = \widehat{\psi}_0^\either + f_n^\either$ and can also be defined as 
	\begin{align*}
		f_n^\either := \lim_{k \rightarrow \infty} \frac{1}{\beta \delta^d |\Lambda_k|} \ln\Phi_n^{\either}(\Lambda_k).
	\end{align*}
	Before proceeding properly into the induction we need the following lemma. 
	
	\begin{lemma}\label{lemma_estimate_truncated_pressure_difference}
		For $n \ge 1$, if  all contours of class at most $n$ the truncated weights are $\tau$-stable then for any $k\le n$
		
		\begin{equation}
			|\widehat{\psi}^\either_{n} - \widehat{\psi}^\either_{k}| \leq \frac{1}{\beta \delta^d} e^{-\frac{\tau}{2}k^{\nicefrac{d-1}{d}}} \quad \text{and } \quad |\widehat{\psi}_{n} - \widehat{\psi}_{k}| \leq \frac{1}{\beta \delta^d} e^{-\frac{\tau}{2}k^{\nicefrac{d-1}{d}}}.
		\end{equation}
	\end{lemma}
	
	\begin{proof}
		We have $|\widehat{\psi}^\either_{n} - \widehat{\psi}^\either_{k}| = |f_{n}^\either - f_k^\either|$. Since all contour of class at most $n$ are $\tau$-stable we know that the cluster expansion for $f_k^\either$ converges for $k \leq n$. We can notice that the clusters $X$ that contributes to $f_{n}^\either - f_k^\either$ must have at least one contour $\gamma$ of class greater than $k$ and thus by the isoperimetric inequality $|\overline{\gamma}| \geq k^{\nicefrac{d-1}{d}}$ which in turn implies that $|\overline{X}| \geq k^{\nicefrac{d-1}{d}}$ and thus by Lemma \ref{lemma_cluster_expansion}
		\begin{equation}\label{eq_inegalite_difference_pression_troncquee_ordre_k_n+1}
			|\widehat{\psi}^\either_{n} - \widehat{\psi}^\either_{k}| \leq \frac{1}{\beta \delta^d} \sum_{\substack{X : 0 \in \overline{X} \\ |\overline{X}| \geq k^{\nicefrac{d-1}{d}}} } |\widehat{\Psi}^\either(X)| \leq \frac{1}{\beta \delta^d} e^{-\frac{\tau}{2}k^{\nicefrac{d-1}{d}}}.
		\end{equation}
		Now we want to have the same kind of estimate for $|\widehat{\psi}_{n} - \widehat{\psi}_{k}|$. If $\widehat{\psi}_{n} = \widehat{\psi}_{n}^\either$ and $\widehat{\psi}_{k} = \widehat{\psi}_{k}^\either$, we get the same estimate since the difference is the same. In the case where $\widehat{\psi}_{n} = \widehat{\psi}_{n}^\either$ and $\widehat{\psi}_{k} = \widehat{\psi}_{k}^{\either^*}$, where $\either \neq \either^*$, we have on one side
		\begin{equation*}
			\widehat{\psi}_{n} - \widehat{\psi}_k = \widehat{\psi}_{n}^\either -\widehat{\psi}_{k}^\either +\widehat{\psi}_{k}^\either - \widehat{\psi}_{k}^{\either^*} \leq \widehat{\psi}_{n}^\either - \widehat{\psi}_{k}^\either
		\end{equation*}
		since by definition $\widehat{\psi}_{k}^\either - \widehat{\psi}_{k}^{\either^*} = \widehat{\psi}_{k}^\either - \widehat{\psi}_{k} \leq 0$. On the other side, we have 
		\begin{equation*}
			\widehat{\psi}_{n} - \widehat{\psi}_k = \widehat{\psi}_{n}^\either -\widehat{\psi}_{n}^{\either^*} +\widehat{\psi}_{n}^{\either^*} - \widehat{\psi}_{k}^{\either^*} \geq 0
		\end{equation*} 
		since $(\widehat{\psi}_{i}^{\either^*})_{i \in \N}$ is increasing and $\widehat{\psi}_{n}^\either -\widehat{\psi}_{n}^{\either^*} \geq 0$ by definition. In any case, we obtain
		\begin{equation*}
			|\widehat{\psi}_{n} - \widehat{\psi}_k | \leq \frac{1}{\beta \delta^d} e^{-\frac{\tau}{2}k^{\nicefrac{d-1}{d}}}.
		\end{equation*}
	\end{proof}
	
	We move on the proof that \eqref{bound_partition_function} holds if $|\Lambda|= n+1$. Note that any contour that appears inside of $\Lambda$ is at most of class $k\leq n$. We say that a contour $\gamma$ is stable if 
	\begin{align*}
		a_n^\either \delta^d |\Int \gamma|^{\nicefrac{1}{d}} \leq \frac{\rho_0 }{16}.
	\end{align*}
	This property is hereditary, in the sense that for all contours $\gamma'$ that can appear inside $\Int \gamma $ are stable as well. Since, we know that all contours of class at most $n$ are $\tau$-stable we can apply Lemma \ref{lemma_estimate_truncated_pressure_difference} and by \eqref{condition_beta_3} we have for any contour $\gamma$ of class $k \leq n$
	\begin{align*}
		a_k^\either \delta^d |\Int \gamma |^{\nicefrac{1}{d}} & = a_{n}^\either \delta^d |\Int \gamma|^{\nicefrac{1}{d}} + (a_k^\either - a_{n}^\either) \delta^d |\Int \gamma|^{\nicefrac{1}{d}} \\
		& \leq a_{n}^\either \delta^d |\Int \gamma|^{\nicefrac{1}{d}} + 2\beta^{-1} k^{\nicefrac{1}{d}}e^{-\nicefrac{\tau k^{\nicefrac{d-1}{d}}}{2}} \\
		& \leq  a_{n}^\either \delta^d |\Int \gamma|^{\nicefrac{1}{d}} + \frac{\rho_0 }{16}.
	\end{align*}
	Therefore, when the contours are stable it implies that $a_k^\either \delta^d |\Int \gamma |^{\nicefrac{1}{d}} \leq \nicefrac{\rho_0 }{8}$ and thus  $\widehat{w}_{\gamma}^\either = w_{\gamma}^\either$.
	In contrast, we would call contours that doesn't satisfy this condition unstable. The stability of a contour depends on the parameter $s$ as it affects the value of $a_n^\either$. Thus we have two cases to consider. The first case is   $a_n^\either=0$. Consequently all contours are stable, therefore according to Theorem \ref{theorem_cluster_expansion} we have 
	\begin{align*}
		Z_\Lambda^\either & = \widehat{Z}_\Lambda^\either = e^{\beta \delta^d \widehat{\psi}_n^\either|\Lambda| \delta^d + \Delta} \\
		& \leq e^{\beta  \delta^d \widehat{\psi}_n^\either|\Lambda| + |\partial_{ext} \Lambda| }
	\end{align*}
	and \eqref{bound_partition_function} is proved. Now let us consider $a_n^\either >0$, in this case some contours must be unstable. Therefore we can partition the configurations that generate among the external contours those that are unstable
	\begin{align*}
		Z_\Lambda^\either & = \sum_{\substack{\Gamma \in \mathcal{C}_{ext}^\either(\Lambda) \\ \text{unstable}}} \int e^{-\beta H_\Lambda(\omega) } \mathbbm{1}_{\{\forall i \in \partial_{int} \Lambda, \sigma(\omega,i)=\either \}} \mathbbm{1}_{\{\Gamma \subset \Gamma_{ext}(\omega)\}} \Pi_{\widehat{\Lambda}}^{s\beta}(d\omega).\\
	\end{align*}
	Similar to what we did in \eqref{eq_decomposition_product_integrals}, we can write each integral as a product of integrals with respect to Poisson point process distribution on different domains using the properties \eqref{energy_local_gamma_1}, \eqref{energy_local_gamma_2} and \eqref{energy_local_gamma_3}. The only difference is that we do consider for the moment only unstable contours and so inside $\Lambda_{ext}$ we have to account for stable contours. Furthermore, these stable contours that cannot encircle any external unstable contour due to the hereditary property of stable contours. 
	\begin{align*}
		Z_\Lambda^\either = \sum_{\substack{\Gamma \in \mathcal{C}_{ext}^\either(\Lambda) \\ \text{unstable}}} Z_{\Lambda_{ext}, stable}^\either \prod_{\gamma \in \Gamma} I_{\gamma} Z_{\Int_0 \gamma}^{(0)} Z_{\Int_1 \gamma}^{(1)},
	\end{align*}
	where $Z_{\Lambda_{ext}, stable}^\either$ denotes the partition function restricted to configurations for which all contours are stable and by construction they are of class at most $n$. Since all those contours are of class at most $n$ they are also $\tau$-stable therefore they can be studied using a convergent cluster expansion according to Theorem \ref{theorem_cluster_expansion} and thus
	\begin{align*}
		Z_{\Lambda_{ext}, stable}^\either & = g_\either^{|\Lambda_{ext}|} \widehat{\Phi}_{ n, stable}^\either(\Lambda_{ext}) \\
		& \leq g_\either^{|\Lambda_{ext}|} e^{\beta \delta^d f_{n, stable}^\either |\Lambda_{ext}| + |\partial_{ext} \Lambda_{ext}|} \\
		& \leq e^{\beta \delta^d (\widehat{\psi}_0^\either +f_{n,stable}^\either) |\Lambda_{ext}| + |\partial_{ext} \Lambda_{ext}|}
	\end{align*}
	where $f_{n,stable}^\either =\lim_{k \to \infty} \frac{1}{\beta |\Lambda_k| \delta^d} \ln\widehat{\Phi}_{n, stable}^\either(\Lambda_{ext})$. According to the induction hypothesis we have that 
	\begin{align*}
		Z_{\Int_0 \gamma}^{(0)} Z_{\Int_1 \gamma}^{(1)} & \leq e^{\beta \delta^d \widehat{\psi}_n|\Int \gamma|} e^{2(|\partial_{ext} \Int_1 \gamma| + |\partial_{ext} \Int_0 \gamma|)} \\
		& \leq e^{\beta \delta^d \widehat{\psi}_n|\Int \gamma|} e^{2|\overline{\gamma}|}.
	\end{align*}
	For any $\Gamma \in \mathcal{C}_{ext}^\either(\Lambda)$ we have $|\partial_{ext} \Lambda_{ext}| \leq |\partial_{ext} \Lambda| + \sum_{\gamma \in \Gamma} |\overline{\gamma}|$, thus we get that
	\begin{align*}
		Z_\Lambda^\either & \leq \sum_{\substack{\Gamma \in \mathcal{C}_{ext}^\either(\Lambda) \\ \text{unstable}}} e^{\beta \delta^d (\widehat{\psi}_0^\either +f_{n,stable}^\either) |\Lambda_{ext}|} e^{ |\partial_{ext} \Lambda| + \sum_{\gamma \in \Gamma} |\overline{\gamma}|} \prod_{\gamma \in \Gamma} I_{\gamma} e^{\beta \delta^d \widehat{\psi}_n |\Int \gamma|} e^{2|\overline{\gamma}|} \\
		& \leq e^{\beta \delta^d \widehat{\psi}_n |\Lambda| + |\partial_{ext} \Lambda|} \sum_{\substack{\Gamma \in \mathcal{C}_{ext}^\either(\Lambda) \\ \text{unstable}}} e^{-\beta \delta^d (\widehat{\psi}_n - \widehat{\psi}_{n, stable}^\either) |\Lambda_{ext}|} \prod_{\gamma \in \Gamma} I_{\gamma} e^{(3-\beta \delta^d \widehat{\psi}_n)|\overline{\gamma}| }
	\end{align*}
	where we define $\widehat{\psi}_{n, stable}^\either :=  \widehat{\psi}_0^\either + f_{n, stable}^\either$. Furthermore when we use the following inequalities $\widehat{\psi}_n \geq \widehat{\psi}_n^\either \geq \widehat{\psi}_0^\either = \frac{\ln(g_\either)}{\beta \delta^d}$ we get that $I_{\gamma} e^{-\beta \delta^d \widehat{\psi}_n |\overline{\gamma}|} \leq e^{-\tau |\overline{\gamma}|}$. Therefore we have
	\begin{align*}
		Z_\Lambda^\either & \leq e^{\beta \delta^d \widehat{\psi}_n |\Lambda| + |\partial_{ext} \Lambda|} \sum_{\substack{\Gamma \in \mathcal{C}_{ext}^\either(\Lambda) \\ \text{unstable}}} e^{-\beta \delta^d (\widehat{\psi}_n - \widehat{\psi}_{n, stable}^\either) |\Lambda_{ext}|} \prod_{\gamma \in \Gamma} e^{-(\tau - 3) |\overline{\gamma}|}.
	\end{align*}
	It remains to prove that the sum is bounded by $e^{|\partial_{ext} \Lambda|}$.  First we note that $\widehat{\psi}_n - \widehat{\psi}_{n, stable}^\either = a_n^\either + f_n^\either - f_{n, stable}^\either$. By construction, the clusters that appear in the cluster expansion of $f_n^\either - f_{n, stable}^\either$ contain at least one unstable contour $\gamma$. Therefore	
	\begin{align*}
		|\overline{\gamma}| \geq |\Int \gamma|^{\nicefrac{d-1}{d}} \geq \left( \frac{\rho_0 }{16 a_n^\either \delta^d}\right)^{(d-1)}
	\end{align*}	
	and by Lemma \ref{lemma_cluster_expansion} and \eqref{condition_beta_5}	
	\begin{align}
		&|f_n^\either - f_{n, stable}^\either| \leq \beta^{-1} \delta^{-d} \exp{(-\max\{(\nicefrac{\rho_0 }{16 a_n^\either \delta^d})^{d-1}, l_0 \} \frac{\tau}{2})} \leq \frac{a_n^\either}{2}.\label{toto}
	\end{align}	
	At the end we obtain
	\begin{align} \label{eq_borne_inf_psi_n_psi_n_stable}
		\widehat{\psi}_n - \widehat{\psi}_{n, \text{stable}}^\either \geq \frac{a_n^\either}{2}.
	\end{align}
	Now let us define new weights $w_\gamma^*$ as follow
	\begin{align*}
		w_{\gamma}^* = \left\{ 
		\begin{aligned}
			&e^{-(\tau-5)|\overline{\gamma}|} &\quad \text{if $\gamma$ is unstable} \\
			&0 & \quad \text{otherwise}.
		\end{aligned}
		\right.
	\end{align*}	
	We denote by $\Phi^*$ the associated polymer development and have
	\begin{equation*} 
		g^* = \lim_{k \rightarrow \infty} \frac{1}{\beta \delta^d |\Lambda_k|} \ln \Phi^*(\Lambda_k).
	\end{equation*}
	For sufficiently large $\beta$ we can assure by Theorem \ref{theorem_cluster_expansion} that it is a convergent cluster expansion. Since all contours that contribute to $g^*$ are all unstable, we obtain an inequality similar to \eqref{toto}
	\begin{align}\label{eq_borne_inf_g_star}
		|g^*| \leq \beta^{-1} \delta^{-d} \exp{(-\max\{(\nicefrac{\rho_0 }{16 a_n^\either \delta^d})^{d-1}, l_0 \} \frac{\tau}{2})} \leq \frac{a_n^\either}{2}.
	\end{align}
	Therefore with \eqref{eq_borne_inf_psi_n_psi_n_stable} and \eqref{eq_borne_inf_g_star} we have $ \widehat{\psi}_n - \widehat{\psi}_n^\either \geq g^*$ and thus
	\begin{align*}
		\sum_{\substack{\Gamma \in \mathcal{C}_{ext}^\either(\Lambda) \\ \text{unstable}}} e^{-\beta \delta^d (\widehat{\psi}_n - \widehat{\psi}_{n, stable}^\either) |\Lambda_{ext}|} \prod_{\gamma \in \Gamma} e^{-(\tau - 3) |\overline{\gamma}|} & \leq  \sum_{\substack{\Gamma \in \mathcal{C}_{ext}^\either(\Lambda) \\ \text{unstable}}} e^{-\beta \delta^d g^* |\Lambda_{ext}|} \prod_{\gamma \in \Gamma} e^{-(\tau - 3) |\overline{\gamma}|}\\
		& \leq e^{-\beta \delta^d g^* |\Lambda|} \sum_{\substack{\Gamma \in \mathcal{C}_{ext}^\either(\Lambda) \\ \text{unstable}}} \prod_{\gamma \in \Gamma} e^{-(\tau - 3) |\overline{\gamma}|} e^{\beta \delta^d g^* (|\overline{\gamma}|+ |\Int \gamma|)}.
	\end{align*}
	By \eqref{eq_borne_inf_g_star} we know that $\beta \delta^d g^* \le 1$, and again with Theorem \ref{theorem_cluster_expansion} we know that $\Phi_{\Int \gamma}^* \geq e^{\beta \delta^d g^* |\Int \gamma| - |\overline{\gamma}|}$ and so	
	\begin{align*}
		\sum_{\substack{\Gamma \in \mathcal{C}_{ext}^\either(\Lambda) \\ \text{unstable}}} e^{-\beta \delta^d (\widehat{\psi}_n - \widehat{\psi}_{n, stable}^\either) |\Lambda_{ext}|} \prod_{\gamma \in \Gamma} e^{-(\tau - 3) |\overline{\gamma}|} & \leq e^{-\beta \delta^d g^* |\Lambda|} \sum_{\substack{\Gamma \in \mathcal{C}_{ext}^\either(\Lambda) \\ \text{unstable}}} \prod_{\gamma \in \Gamma} e^{-(\tau - 5) |\overline{\gamma}|} \Phi_{\Int \gamma}^* \\
		& = e^{-\beta \delta^d  g^* |\Lambda|} \Phi_{\Lambda}^* \\
		& \leq e^{|\partial_{ext} \Lambda|}.
	\end{align*}	
	In summary, we have 
	\begin{equation*}
		Z_\Lambda^\either  \leq e^{\beta \delta^d \widehat{\psi}_n |\Lambda| + 2|\partial_{ext} \Lambda|}
	\end{equation*} 
	which is exactly \eqref{bound_partition_function} in the case where $|\Lambda|=n+1$. If $|\Lambda|\le n$ it is sufficient to notice that $(\widehat{\psi}_n^\either)_{n \in \N}$ is increasing. 
	Let us now prove that \eqref{bound_derivative_partition_function} holds. We start with the case $|\Lambda| = n+1$ and by a similar argument it is true for any smaller $\Lambda$. By a direct computation we have 
	\begin{align*}
		\frac{\partial Z_\Lambda^\either}{\partial z} &= -|\Lambda| \delta^d Z_\Lambda^\either + \frac{1}{z} \int N_{\widehat{\Lambda}}(\omega) e^{-\beta H_\Lambda(\omega)} \mathbbm{1}_{\{\forall i \in \partial_{int} \Lambda, \sigma(\omega, i) = \either \}} \Pi_{\widehat{\Lambda}}^z(d\omega) \\
		&= -|\Lambda| \delta^d Z_\Lambda^\either + \frac{1}{z} E_{P_\Lambda^\either}(N_{\widehat{\Lambda}}) Z_\Lambda^\either.
	\end{align*}
	We can observe that for any configuration $\omega \in \Omega_f$ and for any contour $\gamma \in \Gamma(\omega)$ created by this configuration we have $H_{\overline{\gamma}}(\omega) \geq |\overline{\gamma}_1|\delta^d + \rho_0 |\overline{\gamma}| >0$, thus $H_\Lambda(\omega) \geq 0$. By Donsker-Varadhan inequality for the Kullback-Liebler divergence, denoted by $I(\cdot | \cdot)$, we have
	\begin{align}
		E_{P_\Lambda^\either}(N_{\widehat{\Lambda}}) &\leq I(P_\Lambda^\either | \Pi_{\widehat{\Lambda}}^z) + \ln E_{\Pi_{\widehat{\Lambda}}^z}(e^{N_{\widehat{\Lambda}}}) \nonumber \\
		& \leq \int -\beta H_\Lambda dP_\Lambda^\either - \ln Z_\Lambda^\either + (e-1)z \delta^d |\Lambda| \nonumber \\
		& \leq - \ln Z_\Lambda^\either + (e-1)z \delta^d |\Lambda|. \label{choupi}
	\end{align}
	Furthermore we know that the contours which appear in $|\Lambda|$ are at most of the class $n$. Therefore we know that their truncated weights are $\tau$-stable and by Theorem \ref{theorem_cluster_expansion}  
	\begin{equation}\label{eq_inegalite_inf_fonction_partition_par_cluster}
		Z_\Lambda^\either \geq \widehat{Z}_\Lambda^\either \geq e^{\beta \delta^d \widehat{\psi}_n^\either|\Lambda| - |\partial_{ext} \Lambda|}.
	\end{equation}	
	From inequalities \eqref{eq_inegalite_inf_fonction_partition_par_cluster} and \eqref{choupi} and by using the fact that $(\widehat{\psi}_n^\either)_{n \in \N}$ is increasing we obtain 	
	\begin{equation*}
		E_{P_\Lambda^\either}(N_{\widehat{\Lambda}}) \leq  \left((e-1)z - \beta \widehat{\psi}_0^\either\right) \delta^d |\Lambda| + |\partial_{ext} \Lambda|.
	\end{equation*}	
	At the end we have	
	\begin{equation*}
		\left| \frac{\partial Z_\Lambda^\either}{\partial s} \right| \leq \left[ \left( e + \left|\frac{\widehat{\psi}_0^\either}{s} \right| \right) \beta \delta^d |\Lambda| + \frac{1}{s} |\partial_{ext} \Lambda| \right] Z_\Lambda^\either.
	\end{equation*}
	Now using \eqref{bound_partition_function} we obtain  for $s \in U_\beta$ 
	\begin{equation*}
		\left| \frac{\partial Z_\Lambda^\either}{\partial s} \right| \leq \left( C_1 \beta \delta^d |\Lambda| + C_2 |\partial_{ext} \Lambda| \right) e^{\beta \delta^d \widehat{\psi}_n |\Lambda| + 2|\partial_{ext} \Lambda|}
	\end{equation*}
	and \eqref{bound_derivative_partition_function} is proved. 
	\newline 

	Now let us prove \eqref{truncated_weights_tau_stability} which is the $\tau$-stability of truncated weights for contours of class $n+1$. We consider a contour $\gamma$ of class $n+1$. First of all, we can observe that $\widehat{w}_{\gamma}^\either = 0$ whenever $(\widehat{\psi}_n^{\either^*} - \widehat{\psi}_n^\either) \delta^d |\Int_{\either^*} \gamma|^{\nicefrac{1}{d}} > \nicefrac{\rho_0 }{4}$. So we can assume that 
	\begin{align}
		(\widehat{\psi}_n^{\either^*} - \widehat{\psi}_n^\either) \delta^d  |\Int_{\either^*} \gamma|^{\nicefrac{1}{d}} \leq \frac{\rho_0}{4}. \label{eq_hypothesis_cutoff_value} 
	\end{align}
	Since $|\Int \gamma| = n+1$  we can apply the induction hypothesis on the partition functions that appears in the truncated weights particularly we can use \eqref{bound_partition_function} and have 
	\begin{align} \label{ineq_induction_hypothesis_partition_function}
		Z_{\Int_{\either^*} \gamma}^{\either^*} \leq e^{\beta \delta^d \widehat{\psi}_n |\Int_{\either^*} \gamma| + 2|\partial_{ext} \Int_{\either^*} \gamma|}.
	\end{align}
	By combining the previous inequalities \eqref{ineq_induction_hypothesis_partition_function} and \eqref{eq_inegalite_inf_fonction_partition_par_cluster} we have
	\begin{align*}
		\frac{Z_{\Int_{\either^*} \gamma}^{\either^*}}{Z_{\Int_{\either^*} \gamma}^{\either}} & \leq e^{\beta \delta^d (\widehat{\psi}_n - \widehat{\psi}_n^\either) |\Int_{\either^*} \gamma| + 3|\partial_{ext} \Int_{\either^*} \gamma|} \\
		& \leq e^{\beta \delta^d (\widehat{\psi}_n - \widehat{\psi}_n^\either) |\Int_{\either^*} \gamma|^{\nicefrac{1}{d}} |\Int_{\either^*} \gamma|^{\nicefrac{(d-1)}{d}} + 3|\partial_{ext} \Int_{\either^*} \gamma|}.
	\end{align*}
	Furthermore, applying  hypothesis \eqref{eq_hypothesis_cutoff_value} and the isoperimetric inequality  we have 
	\begin{align}
		\frac{Z_{\Int_{\either^*} \gamma}^{\either^*}}{Z_{\Int_{\either^*} \gamma}^{\either}} & \leq e^{ \nicefrac{\beta  \rho_0 }{4} |\Int_{\either^*} \gamma|^{\nicefrac{(d-1)}{d}} + 3|\partial_{ext} \Int_{\either^*} \gamma|}\nonumber \\
		& \leq e^{ (\frac{1}{4} \beta \rho_0  + 3) |\partial_{ext} \Int_{\either^*} \gamma|}. \label{ineq_upper_bound_ratio_partition_function}
	\end{align} 
	Therefore with Proposition \ref{prop_peierls}, the upper bound on the ratio of partition functions in inequality \eqref{ineq_upper_bound_ratio_partition_function} and the fact that $|\partial_{ext} \Int_{\either^*} \gamma| \leq |\overline{\gamma}|$ we have
	\begin{align*}
		\widehat{w}_{\gamma}^\either  &\leq e^{-(\beta \rho_0 - 2)|\overline{\gamma}|} e^{(\frac{1}{4} \beta \rho_0  + 3) |\overline{\gamma}|} \\
		&\leq e^{-(\frac{3}{4} \beta \rho_0  - 5)
			|\overline{\gamma}|} \\
		&\leq e^{-\tau |\overline{\gamma}|}
	\end{align*}
	and the weight $\widehat{w}_{\gamma}^\either$ is $\tau$-stable. 	Let us now show that \eqref{bound_truncated_weight_derivative} holds for a contour $\gamma$ of class $n+1$. Similar to the proof of \eqref{truncated_weights_tau_stability} we consider only the case when $(\widehat{\psi}_n^{\either^*} - \widehat{\psi}_n^\either) \delta^d |\Int_{\either^*} \gamma|^{\nicefrac{1}{d}} \leq \nicefrac{\rho_0 }{4}$. By a direct computation	
	\begin{align*}
		\frac{\partial \widehat{w}_{\gamma}^\either}{\partial z} = & \left( -\frac{|\overline{\gamma}|}{g_\either} \frac{\partial g_\either}{\partial z} I_{\gamma} + \frac{\partial I_{\gamma}}{\partial z}\right) \left(g_\either^{-|\overline{\gamma}|} \kappa
		\frac{Z^{\either^*}_{\Int_{\either^*} \gamma}}{Z^\either_{\Int_{\either^*} \gamma}}\right) + \frac{\partial}{\partial z} \left( \kappa
		\frac{Z^{\either^*}_{\Int_{\either^*} \gamma}}{Z^{\either}_{\Int_{\either^*} \gamma}}\right) g_\either^{-|\overline{\gamma}|} I_{\gamma}.
	\end{align*}
	The first term of the derivative can be bounded in a similar fashion as in \eqref{ineq_derivative_weights_class0} using the facts that $s \in U_\beta$ and by Proposition \ref{prop_peierls}
	\begin{equation}\label{ineq_bound_first_term_of_derivative}
		\left| -\frac{|\overline{\gamma}|}{g_\either} \frac{\partial g_\either}{\partial z} I_{\gamma} + \frac{\partial I_{\gamma}}{\partial z}\right| \left(g_\either^{-|\overline{\gamma}|} \kappa
		\frac{Z^{\either^*}_{\Int_{\either^*} \gamma}}{Z^\either_{\Int_{\either^*} \gamma}}\right) \leq (2+2K) \delta^d |\overline{\gamma}| e^{-\tau |\overline{\gamma}|}.
	\end{equation}	
	For the other term, the derivative of the cut-off function satisfies	
	\begin{align*}
		\left| \frac{\partial \kappa}{\partial z} \right| \leq \left( \left|\frac{\partial \widehat{\psi}_n^\either}{\partial z}\right| + \left|\frac{\partial \widehat{\psi}_n^{\either^*}}{\partial z}\right| \right) \delta^d |\Int_{\either^*} \gamma|^{\frac{1}{d}} \| \kappa' \|.
	\end{align*}
	For $s \in U_\beta$ and by \eqref{bound_terme_perturbatif_pression} we have 	
	\begin{equation*}
		\frac{\partial \psi_n^\either}{\partial z} = \frac{1}{\beta \delta^d g_\either} \frac{\partial g_\either}{\partial z } + \frac{1}{\beta} \frac{\partial f_n^\either}{\partial s} \implies  \left| \frac{\partial \psi_n^\either}{\partial z} \right| \leq \frac{2}{\beta}.
	\end{equation*}	
	Therefore with the isoperimetric inequality we have 
	\begin{equation}\label{ineq_bound_derivative_cutoff}
		\left| \frac{\partial \kappa}{\partial z} \right| \leq \frac{4}{\beta} \delta^d |\overline{\gamma}| \| \kappa' \|.
	\end{equation}
	For the derivative of the ratio of partition functions	\begin{equation*}
		\left| \frac{\partial}{\partial z} \frac{Z_{\Int_{\either^*}\gamma}^{\either^*}}{Z_{\Int_{\either^*}\gamma}^{\either}} \right| \leq \left|\frac{\nicefrac{\partial Z_{\Int_{\either^*}\gamma}^{\either^*}}{\partial z}}{Z_{\Int_{\either^*}\gamma}^{\either}} \right| + \left|\frac{\nicefrac{\partial Z_{\Int_{\either^*}\gamma}^{\either}}{\partial z}}{Z_{\Int_{\either^*}\gamma}^{\either}} \right| \frac{Z_{\Int_{\either^*}\gamma}^{\either^*}}{Z_{\Int_{\either^*}\gamma}^{\either}}.
	\end{equation*}
	Similarly to how we bounded the ratio of partition function and using \eqref{bound_derivative_partition_function} we have for $\either^* \in \{0,1\}$
	\begin{equation*}
		\left|\frac{\nicefrac{\partial Z_{\Int_{\either^*}\gamma}^{\either^*}}{\partial z}}{Z_{\Int_{\either^*}\gamma}^{\either}} \right| \leq \left(C_1 \delta^d |\Int_{\either^*} \gamma| + \frac{C_2}{\beta} |\partial_{ext}\Int_{\either^*} \gamma| \right) e^{(\nicefrac{1}{4}\beta \rho_0  +3)|\overline{\gamma}|}.
	\end{equation*}
	Therefore using the isoperimetric inequality and the fact that $|\partial_{ext}\Int_{\either^*} \gamma| \leq |\overline{\gamma}|$ we have
	\begin{align}\label{ineq_bound_derivative_partition_function_ratio}
		\left| \frac{\partial}{\partial z} \frac{Z_{\Int_{\either^*}\gamma}^{\either^*}}{Z_{\Int_{\either^*}\gamma}^{\either}} \right|  \leq \left( C_1 \delta^d + \frac{C_2}{\beta} \right)  |\overline{\gamma}|^{\nicefrac{d}{(d-1)}} e^{(\nicefrac{1}{2}\beta \rho_0  +6)|\overline{\gamma}|}.
	\end{align}
	When combining inequalities \eqref{ineq_bound_first_term_of_derivative}, \eqref{ineq_bound_derivative_cutoff}, \eqref{ineq_bound_derivative_partition_function_ratio} and Proposition \ref{prop_peierls} we have 
	\begin{align*}
		\left| \frac{\partial \widehat{w}_{\gamma}^\either}{\partial s} \right| \leq \left((2+2K)\delta^d \beta + 4 \delta^d \| \kappa' \| +  C_1\beta \delta^d + C_2 \right) |\overline{\gamma}|^{\nicefrac{d}{(d-1)}} e^{-\tau |\overline{\gamma}|}.
	\end{align*}
	To finish the proof let us show that \eqref{implication_bound_a_n_to_stability} holds at the order $n+1$. Since we have proven this far that the truncated weights of class at most $n+1$ are $\tau$-stable we can apply use Lemma \ref{lemma_estimate_truncated_pressure_difference}. Let $\gamma$ of class $n+1$ if we have $a_{n+1}^\either \delta^d |\Int \gamma|^{\nicefrac{1}{d}} \leq \frac{\rho_0 }{16}$ then by definition of truncated weights we would have $\widehat{w}_{\gamma}^\either = w_{\gamma}^\either$. Now let's consider a contour $\gamma$ of class $k\leq n$, according to Lemma \ref{lemma_estimate_truncated_pressure_difference} and \eqref{condition_beta_3} we have 
	\begin{align*}
		a_k^\either \delta^d |\Int \gamma|^{\nicefrac{1}{d}} & = a_{n+1}^\either \delta^d |\Int \gamma|^{\nicefrac{1}{d}} + (a_k^\either - a_{n+1}^\either) \delta^d |\Int \gamma|^{\nicefrac{1}{d}} \\
		& \leq a_{n+1}^\either \delta^d |\Int \gamma|^{\nicefrac{1}{d}} + 2\beta^{-1} k^{\nicefrac{1}{d}}e^{-\nicefrac{\tau k^{\nicefrac{d-1}{d}}}{2}} \\
		& \leq  a_{n+1}^\either \delta^d |\Int \gamma|^{\nicefrac{1}{d}} + \frac{\rho_0 }{16}.
	\end{align*}
	Therefore if $a_{n+1}^\either \delta^d |\Int \gamma|^{\nicefrac{1}{d}} \leq \nicefrac{\rho_0 }{16}$ it implies that $a_k^\either \delta^d |\Int \gamma|^{\nicefrac{1}{d}} \leq \nicefrac{\rho_0}{8}$ which in turn would imply $\widehat{w}_{\gamma}^\either = w_{\gamma}^\either$ by definition of the truncated weights.

	\printbibliography
	
\end{document}